\documentclass[10pt, a4paper]{article}

\usepackage[applemac]{inputenc}
\usepackage[T1]{fontenc}

\RequirePackage{amsthm,amsmath,amsfonts,amssymb}
\RequirePackage[]{natbib}
\RequirePackage[colorlinks,allcolors=blue]{hyperref}
\RequirePackage{graphicx}
\RequirePackage{enumerate}
\RequirePackage{booktabs}
\usepackage{graphics, subfigure}
\usepackage{framed}
\usepackage{url}
\usepackage{placeins}
\usepackage[normalem]{ulem} 

\newtheorem{theorem}{Theorem}[section]
\newtheorem{remark}{Remark}

\newtheorem{proposition}[theorem]{Proposition}
\newtheorem{corollary}[theorem]{Corollary}
\newtheorem{definition}[theorem]{Definition}

\renewcommand{\P}{\mathbb{P}}
\renewcommand{\d}{\mathrm{d}}
\DeclareMathOperator{\dlim}{\mathrm{dlim}}

\DeclareMathOperator{\argmin}{\mathrm{argmin}}
\DeclareMathOperator{\E}{\mathbb{E}}
\DeclareMathOperator{\Z}{\mathbb{Z}}
\DeclareMathOperator{\R}{\mathbb{R}}
\DeclareMathOperator{\F}{\mathcal{F}}

\DeclareMathOperator{\bfZ}{\mathbf{Z}}
\DeclareMathOperator{\bfu}{\mathbf{u}}
\DeclareMathOperator{\bfa}{\mathbf{a}}

\renewcommand{\bfu}{\mathbf{u}}
\DeclareMathOperator{\bfx}{\mathbf{x}}
\DeclareMathOperator{\bfX}{\mathbf{X}}
\DeclareMathOperator{\bfY}{\mathbf{Y}}
\DeclareMathOperator{\bfH}{\mathbf{H}}

\DeclareMathOperator{\bfB}{\mathbf{B}}
\DeclareMathOperator{\bfN}{\mathbf{N}}
\DeclareMathOperator{\N}{\mathbb{N}}
\DeclareMathOperator{\bfeta}{\mathbf\eta}
\DeclareMathOperator{\bfw}{\mathbf w}
\DeclareMathOperator{\bfepsilon}{\mathbf \varepsilon}

\DeclareMathOperator{\Var}{\mathrm{Var}}
\DeclareMathOperator{\Cov}{\mathrm{Cov}}

\renewcommand{\vec}{\mathrm{vec}}
\newcommand{\diag}{\mathrm{diag}}

\usepackage{color}
\definecolor{light-gray}{gray}{0.9}

\begin{document}
\title{An estimation procedure for the Hawkes process}
\author{Matthias Kirchner}
 \date{Revised version: January 01, 2017}

\maketitle
\begin{abstract}
In this paper, we present a nonparametric estimation procedure for the multivariate Hawkes point process. The timeline is cut into bins and---for each component process---the number of points in each bin is counted. {As a consequence of earlier results in \citet{kirchner16c},} the distribution of the resulting `bin-count sequences' can be approximated by an integer-valued autoregressive model known as the (multivariate) INAR($p$) model. We represent the INAR($p$) model as a standard vector-valued linear autoregressive time series with white-noise innovations (VAR($p$)). 
We establish consistency and asymptotic normality for conditional least-squares estimation of the VAR($p$), respectively, the INAR($p$) model. After appropriate scaling, these time series estimates yield estimates for the underlying multivariate Hawkes process as well as {corresponding variance estimates}.  {The estimator depends on a bin-size $\Delta$ and a support $s$. We discuss the impact and the choice of these parameters. }All results are presented in such a way that computer implementation, e.g., in {\tt R}, is straightforward. Simulation studies confirm the effectiveness of our estimation procedure. {In the second part of the paper,} we present a data example where the method is applied to bivariate event-streams in financial limit-order-book data. We fit a bivariate Hawkes model on the joint process of limit and market order arrivals. The analysis exhibits a remarkably asymmetric relation between the two component processes: incoming market orders excite the limit-order flow heavily whereas the market-order flow is hardly affected by incoming limit orders. {For the estimated excitement-functions, we observe power-law shapes, inhibitory effects for lags under $0.003\sec$, second periodicities, and local maxima at $0.01\sec$, $0.1\sec$, and $0.5\sec$.}
Keywords: Hawkes process; estimation; integer-valued autoregressive time series; contagion model; intraday financial econometrics
\end{abstract}

\section{Introduction}\label{Introduction}
In this paper, we introduce a nonparametric estimation procedure for the multivariate Hawkes point process; see Definition~\ref{estimator} for the formal definition and Figure \ref{fig2} for an illustrative summary of the main results. The Hawkes process is a model for event streams. Its alternative name, `selfexciting point process', stems from the fact that any event has the potential to generate new events in the future. Our estimator gives substantial information on this excitement: nonmonotonicities or regime switches in the excitement of the fitted Hawkes model can be detected; the estimates may also help with the choice of parametric excitement-functions. The asymptotic distribution of the estimator can be derived so that confidence bounds are at hand. Also note that the presented estimation method is numerically less problematic than the standard likelihood-approach.
Last but not least, the figures generated from the estimation results are a graphical tool for representing large univariate and multivariate event data sets in a compact and at the same time informative way. In particular, the estimation results can be interpreted as measures for interaction and stability 
of empirical event-streams. This will be highlighted in the data example at the end of the paper
 where we apply the estimation procedure to the order arrival times in an electronic market. 
\par The Hawkes process was introduced in \citet{hawkes71a, hawkes71b} as a model for event data from contagious processes. 
Theoretical cornerstones of the model are
 \citet{
 hawkes74},
 \citet{
bremaud96,
bremaud01},
\citet{liniger09}
and
\citet{errais10}.
For a textbook reference that covers many aspects of the Hawkes process; see \citet{daley03}. The main theoretical reference for the following presentation is our own contribution \citet{kirchner16c}, where we show that Hawkes processes can be approximated by certain discrete-time models. 
\par By the omnipresence of `event'-type data, the Hawkes process has become a popular model in many different contexts such as geology, e.g., earthquake modeling in \citet{ogata88}, internet traffic, e.g., youTube clicks in \citet{crane08}, biology, e.g., genome analysis in \citet{reynaud10}, sociology, e.g., crime data in \citet{mohler11}, or medicine, e.g., virus spreading in \citet{kim11}. A most active  area of scientific activity today is financial econometrics with applications of Hawkes processes to the modeling of credit defaults in \citet{errais10}, extreme daily returns in \citet{liniger11}, market contagion in \citet{sahalia11} and numerous applications to limit-order-book modeling such as high-frequency price jumps in \citet{bacry11b} and \citet{chavez12}, order arrivals in \citet{bacry11a}, or joint models for orders and prices on a market microstructure level in \citet{muzy13}. Early publications applying the Hawkes model in the financial context are \citet{bowsher02}, \citet{chavez05} and \citet{mcneil05}. 
\par The paper is organized as follows: Section 2 explains how Hawkes processes can be approximated by specific integer-valued time series and how this approximation yields an estimation procedure. Section 3 defines the new Hawkes estimator formally, discusses its properties, {and compares it to alternative estimation methods}. Section 4 refines the procedure by giving methods for a reasonable choice of the estimation parameters.
Section 5 presents the data example where the ideas of the paper are applied to the analysis of intraday financial data. The last section concludes with a discussion on the implications of the presented results. 
Appendix \ref{proofs} contains proofs. Large parts of the paper are accompanied by examples with simulated data: in favor of a linear reading flow, we directly illustrate all new concepts with such examples---instead of devoting a separate section to simulations.

\begin{figure}
\includegraphics[width = \textwidth]{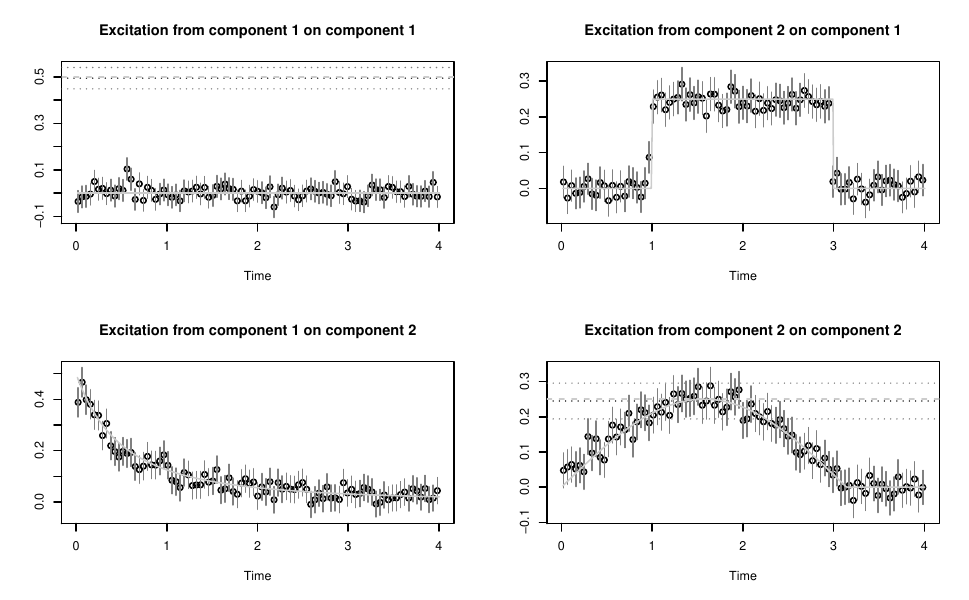}
\caption{Summary of the main result of the paper. From the bivariate Hawkes model presented in Figure~\ref{fig1}, around 100\,000 events in each component are simulated. From this single large sample, we calculate the estimator from Definition~\ref{estimator}. The black circles refer to estimated values of the excitement functions. The horizontal black dotted lines in the diagonal panels refer to the corresponding estimated baseline-intensity components. The vertical grey lines as well as the dotted horizontal grey lines  refer to marginal 95\%-confidence intervals; see Remark~\ref{cov}. All solid and dashed lightish-grey lines refer to the true underlying parameters; compare with Figure~\ref{fig1}. Eyeball examination shows that the estimation method approximates the form of the true excitement functions well. Also the
nonmonotonicities and the jumps are reproduced. The coverage rates of the confidence intervals seem just about right. There is no obvious bias. For a more quantitative analysis of the estimation method; see Section~\ref{simulation_study} and Figure~\ref{fig15}.}
\label{fig2}
\end{figure}

\section{Approximation of Hawkes processes}\label{The_Method}
In this section, after defining the Hawkes process
we introduce autoregressive integer-valued time series. 
We clarify how this model approximates the Hawkes model and how this approximation yields an estimation procedure. 

\subsection{The Hawkes process}\label{Hawkes_Process}
From a geometric point of view, a (univariate) Hawkes process specifies a distribution of points on a line. Typically, the line is interpreted as `time' and the points as `events'. Selfexciting point process is the common alternative name for the Hawkes process. It highlights the basic idea of the model: given an event, the intensity---the expected number of events in one time unit---shoots up (`selfexcites') and then decays (`forgets its past gradually'). The shape of this decay is specified by a function, namely the excitement function.
The definition and the proof of existence of a Hawkes process are subtle matters. For rigorous theoretical foundation, we refer to \citet{liniger09}, Chapter 6. We assume a basic underlying probability space $\left(\Omega, \P, \mathcal{F}\right)$, complete and rich enough to carry all random variables involved. On this probability space, we define stochastic point-sets $\mathcal{P}\subset \R$ of the form $\mathcal{P}=\{\dots, T_{-1}, T_0, T_1, \dots\}$ with $T_k\leq T_{k+1}$, $k \in\Z$, having almost surely no limit points. Furthermore, we assume that the $\sigma$-algebras 
$$
\mathcal{H}_t^{\mathcal{P}} :=\sigma\Bigg(\bigg\{\omega\in\Omega:\#\left(\mathcal{P(\omega)}\cap(a,b]\right) = n:\; n\in\mathbb{N}_0, \; a < b \leq t\bigg\}\Bigg),\quad t\in\R,
$$
 are subsets of $\mathcal{F}$. By setting 
$$
N_{\mathcal{P}}(A):=\#\left(\mathcal{P}\cap A\right),\quad A\in\mathcal{B}(\R),
$$
any stochastic point-set $\mathcal{P}$ defines a random measure $N_{\mathcal{P}}$ on $\mathcal{B}(\R)$, the Borel sets of $\R$. 
At this point, we drop the $\mathcal{P}$ index; the set $\mathcal{P}$ is completely specified by $N:=N_{\mathcal{P}}$.
In this paper, we call a random measure $N$ of this kind \emph{point process} and we call the filtration $\left(\mathcal{H}_t^{N}\right):=\left(\mathcal{H}_t^{\mathcal{P}}\right)$ \emph{history} of the point process.     
  The \emph{conditional intensity} of a point process $N$ is
\begin{align}
\Lambda_{N}(t) &:=\lim\limits_{\delta\downarrow 0}\frac{\E\Big[N\left((t, t+\delta] \right)|\mathcal{H}_{t}^N\Big]}{\delta},\quad t\in\R. \label{conditional_intensity}
\end{align}

\begin{figure}
\begin{center}
\subfigure[Model parameters of a bivariate Hawkes process. The solid lines refer to the excitement function $H=(h_{i,j})$. It consists of the two selfexcitement functions $h_{1,1}(t) \equiv 0$ and $ h_{2,2}(t) = 1_{t\leq \pi} 0.25 \sin(t)$ as well as the two crossexcitement functions $h_{1,2}(t) \equiv 1_{1<t\leq 3} 0.25 $ and $h_{2,1}(t) = 0.5(1+t)^{-2}$.
The dashed lines in the diagonal panels refer to the two components of the baseline intensity $\bfeta = (0.5, 0.25)$. The functions are chosen quite extreme for the sake of demonstration of the estimation method; see Figure~\ref{fig2}.]{\includegraphics[width = 0.9\textwidth]{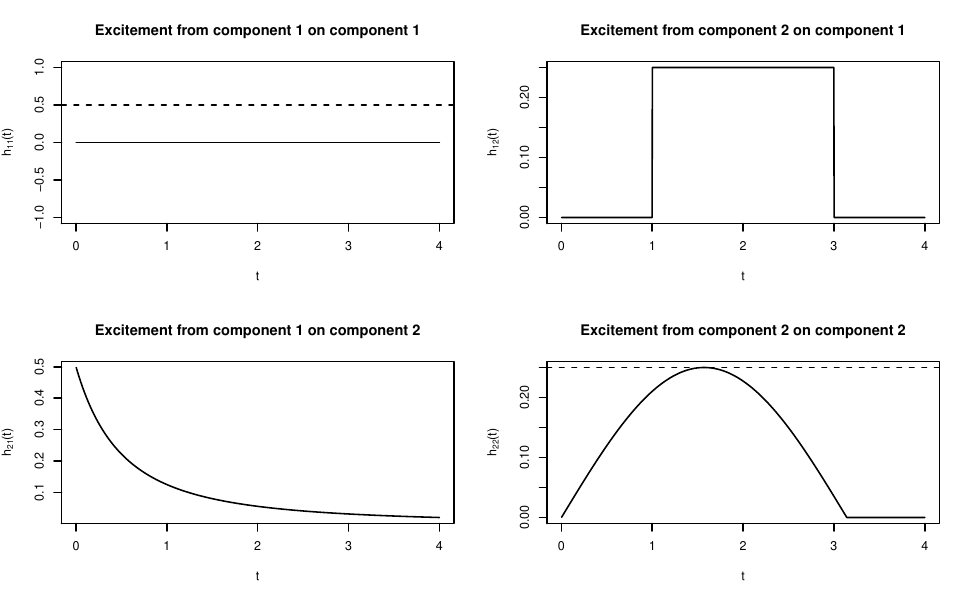}}
\subfigure[A realization of the two components of the process starting at time 0. The vertical lines refer to the events, the greyish solid lines refer to the realized conditional-intensity components, and the dashed lines refer to the baseline-intensity components. The crossexcitement from component 1 on component 2 and also the delayed rectangle impuls impact from component 2 on component 1 are particularly visible.]{\includegraphics[width = 0.9\textwidth]{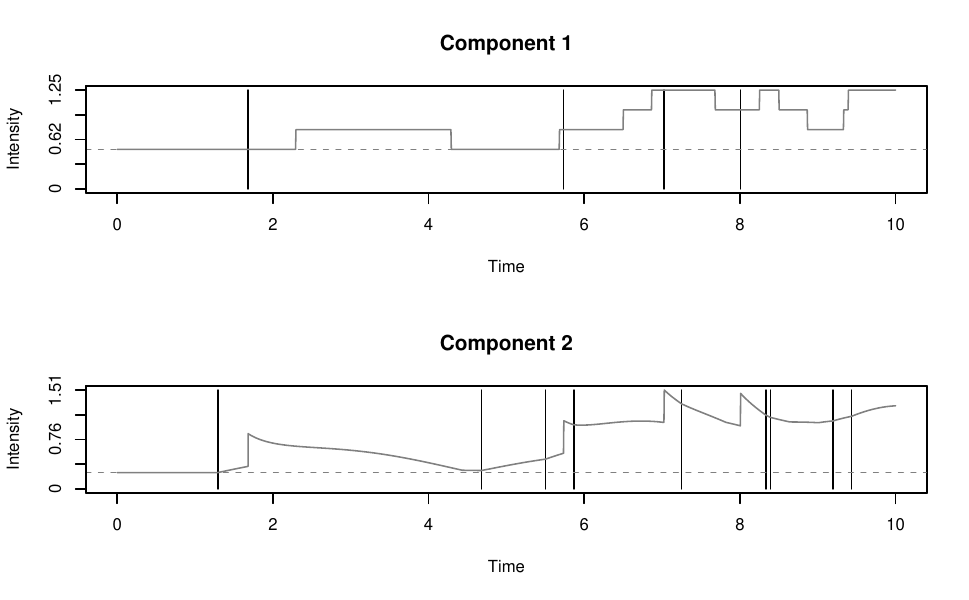}}
\caption{Illustration of a bivariate Hawkes process as described in Section~\ref{Hawkes_Process}. The upper panel shows the model parameters. The lower panel shows a realization.}
\label{fig1}
\end{center}
\end{figure}

A \emph{Hawkes process} is a stationary point process $N$ with conditional intensity
\begin{align}
\Lambda_N(t)&=\eta+\int\limits_{-\infty}^t h(t-s)  N\left(\mathrm{d}s\right),\quad t\in\R.\label{Hawkes_intensity}
\end{align}
The constant $\eta\geq 0$ is called \emph{baseline intensity}, and the function
$h: \R_{\geq0}  \rightarrow \R_{\geq 0}$,  
measurable, is called \emph{excitement function}. Existence-conditions are discussed below. 
\par For $d\in\N$, a $d$\emph{-variate Hawkes process} $\mathbf{N} $ is a process with $d$ point processes on $\R$ as components, i.e.,
$
\mathbf{N} = \left(N^{(1)},\dots ,N^{(d)}\right)^\top.
$
Each component process counts points from random point-sets $\mathcal{P}_1\subset\R,\dots,\mathcal{P}_d\subset\R$. In this multivariate setup, the counting processes $N^{(k)},\,k=1,\dots ,d,$ do not only selfexcite but in general also interact with each other (`crossexcite'). The baseline intensity $\bfeta$ is a $d$-variate vector in $\R_{\geq 0}^d$ and the excitement function is a measurable $d\times d$ matrix-valued function $H=\left(h_{i,j}\right)_{1\leq i,j\leq d}: \R_{\geq0}  \rightarrow \R_{\geq 0}^{d\times d}$. The \emph{conditional intensity of a $d$-variate Hawkes process} is $ \R_{\geq0}^d$-valued with
\begin{align}
\mathbf{\Lambda}_{\bfN}(t)&:=\lim\limits_{\delta\downarrow 0} \frac{\E\Big[\mathbf{N}((t, t+\delta])|\mathcal{H}_{t}^{\bfN}\Big]}{\delta}=\bfeta+\int\limits_{-\infty}^{t}H(t-s) \mathbf{N}\left(\mathrm{d}s\right),
\quad t\in\R, \label{multivariate_intensity}
\end{align}
where, for $i= 1,\dots, d$,
\begin{equation}
\left(\,\int\limits_{-\infty}^{t}H(t-s) \mathbf{N}\left(\mathrm{d}s\right)\right)_i := \left(\sum\limits_{j = 1}^d \int\limits_{-\infty}^{t}h_{i,j}(t-s) N^{(j)}\left(\mathrm{d}s\right)\right)_i
\end{equation}
and
$\mathcal{H}_{t}^{\mathcal{\bfN}} :=\sigma\Bigg(\bigg\{\omega\in\Omega:\, \bfN\big((a,b]\big) = \mathbf{n}\bigg\},\,\mathbf{n}\in\N_0^d,\ a<b\leq t\Bigg).
$
In other words, the entry $h_{i,j}(t)$ of the matrix $H(t)$ denotes the effect of any event $T^{(j)}_k\in\mathcal{P}_j$ in component $j$ on the intensity of component $i$ at time $T^{(j)}_k+t$. See Figure~\ref{fig1} for an example of a bivariate Hawkes process.
In \citet{hawkes71b}, we find the following {sufficient condition for existence}: if
\begin{align}
\mathrm{spr}(K):=\max\Big\{|k|: k \text{ eigenvalue of matrix }K\Big\} < 1,\label{stability_criterion}
\end{align}
where $K :=\left(\int\limits_0^\infty h_{i,j}(t) \mathrm{d}t\right)_{1\leq i,j\leq d}$,
then a process with conditional intensity as in \eqref{multivariate_intensity} exists. 
The matrix $K$ in \eqref{stability_criterion} is sometimes referred to as \emph{branching matrix} and the entries of $K$ as \emph{branching coefficients}. These terms reflect an alternative view on the process as a special cluster process \citep{hawkes74}:
\par In each of the components of a $d$-variate Hawkes process, we observe cluster centers that stem from independent homogeneous Poisson processes with rates $\eta_1, \dots , \eta_d$. These cluster centers are also called \emph{immigrants} or \emph{exogenous events}. Such an immigrant $I^{(j)}\in\R$ in component $j$ triggers $d$ inhomogeneous Poisson processes in components $i=1,\dots ,d$ with intensities $h_{i,j}\left(\cdot - I^{(j)}\right),\, i=1,\dots,d$. And each of these new points again produces $d$ inhomogeneous Poisson processes in a similar way, so that the clusters are built up as a cascade of inhomogeneous Poisson processes. The non-immigrant events are called \emph{offspring} or \emph{endogenous events}. Disregarding the time component and only considering this immigrant--offspring structure, one actually has a branching process with immigration, where the number of direct offspring in component $i$ from an event in component $j$ is Pois\big($K_{ij}$\big) distributed. 

\subsection{Parametrization and estimation of Hawkes processes}
In most cases, the data analyst's choice of the excitement function $H$ of a Hawkes process is a somewhat arbitrary parametric function---the main decision being between exponential functions or power-law functions. The function parameters are then estimated via standard likelihood maximization. Power-law decay of the excitement functions often turns out to be more `realistic' in applications; exponential decay yields a likelihood that is numerically easier to handle by recursive representation; see \citet{ogata88}. In addition, exponential excitement functions are mathematically attractive because they yield a Markovian structure for the conditional intensity; see \citet{errais10}.
Even if the choice between exponential and power-law decay is handled carefully, these two functional families cannot catch regime switches or nonmonotonicities of excitement functions as in Figure~\ref{fig1}. So it seems important to develop methods that can identify shapes of excitement in data with less stringent assumptions. Another motivation for our research on estimation of the Hawkes model stems from numerical issues---especially encountered in the multivariate case.
 A third gap that we aim to close with our paper is the derivation of the asymptotic distribution of the estimates. 
\par {
Alternative estimation methods for the Hawkes process have been introduced in \citet{lewis11}, \citet{lemonnier14} , \citet{reynaud14}, \citet{alfonsi15}, and \citet{hansen15}. In particular, the method developed in \citet{bacry11b, bacry14a}, and \citet{bacry15b} is similar to ours and can be interpreted in our approximation framework. We will discuss these alternative estimation approaches in Section \ref{alternative_estimation_methods}.}

\subsection{Intuition of the approximation}\label{intuition}
The main idea is simple: given a (possibly multivariate) Hawkes process, we divide the time line into bins of size $\Delta>0$ and count the number of events in each bin (for each component). These `bin counts' form an $\N_0$-valued stochastic sequence ($\N_0^d$-valued in the $d$-variate case). The distribution of this sequence can be approximated by a well-known time series model. We present the heuristics behind the approximation in the case of a univariate Hawkes process $N$ with baseline intensity $\eta>0$ and excitement function $h$ with $\int h\mathrm{d}t <1$. For some $\Delta>0$, we define the bin-counts $
\tilde{X}^{(\Delta)}_{n}:
=N\big(((n-1)\Delta , n\Delta]\big),\, n\in\Z .
$
We want to argue that for small $\Delta >0$ and large $p\in\N$, we have that
\begin{align}
\E\left[\tilde{X}^{(\Delta)}_{n}\Big|\sigma\left(\tilde{X}^{(\Delta)}_{n-1}, \tilde{X}^{(\Delta)}_{n-2}, \dots\right)\right] \approx \Delta\eta +  \sum\limits_{k=1}^p \Delta h(\Delta k)\tilde{X}^{(\Delta)}_{n-k},\quad n\in\Z.\label{approx_bin_count_eq}
\end{align}
We divide the approximation above in three separate approximation-steps:
 \begin{eqnarray}
\E\left[\tilde{X}^{(\Delta)}_{n}|\mathcal{H}^N_{(n-1)\Delta}\right] \Bigg(&\stackrel{\eqref{conditional_intensity}}{=}& \int\limits_{(n-1)\Delta}^{n\Delta}\E\left[\Lambda(t)|\mathcal{H}^N_{(n-1)\Delta}\right]\mathrm{d}t\Bigg)\nonumber\\
&\stackrel{\eqref{Hawkes_intensity}}{\approx}& \Delta\eta + \Delta\int\limits_{-\infty}^{(n-1)\Delta}h(n\Delta-u) N\left(\mathrm{d}u\right)\label{distributional_error}
\\
&\approx& \Delta\eta + \Delta\int\limits_{(n-p-1)\Delta}^{(n-1)\Delta}h(n\Delta-u) N\left(\mathrm{d}u\right)\label{cutoff_error}\\
&\approx& \Delta\eta + \sum\limits_{k=1}^p \Delta h(\Delta k)\tilde{X}^{(\Delta)}_{n-k},\quad n\in\Z.\label{discretization_error}
\end{eqnarray}
The estimator we are about to present ignores the three approximations above and treats them as equalities. In doing so, we make a distributional error~\eqref{distributional_error}, a cut-off error~\eqref{cutoff_error}, and a discretization error~\eqref{discretization_error}.
{The term \emph{distributional error} might demand further explanation: in~\eqref{distributional_error}, we treat the conditional intensity $\Lambda$ as constant over $((n-1)\Delta ,n\Delta]$. This is not true in general as $h$ is typically not (piecewise) constant. In addition---and more importantly---\eqref{distributional_error} ignores the influence of possible events in the bin $((n-1)\Delta , n\Delta]$ on $\Lambda$. As an example, suppose we observe two events in a bin. In the original Hawkes model, the second of these events may very well be a result of the first event. But in the approximating model, we ignore this possibility and explain both of these events by events in earlier bins or by the constant term. 
}
\par There is an integer-valued time series that solves the approximative bin-count equation \eqref{approx_bin_count_eq} to the point: the integer-valued autoregressive model of order $p\in\N$, the INAR($p$) model. {The three different approximation errors \eqref{distributional_error},~\eqref{cutoff_error}, and~\eqref{discretization_error}, contribute to the bias of our estimation method in different ways. We discuss these effects in Sections~\ref{choice_of_support} and~\ref{choice_of_bin_size}.}

\subsection{The INAR($p$) model}
The INAR($p$) process was first proposed by \citet{li91} as a time series model for count data. For the history and an exhaustive collection of properties of the model; see \citet{marques05}. For a textbook reference; see \citet{fokianos01}. The main idea of the construction is to manipulate the standard system of autoregressive difference-equations `$X_ n -\sum \alpha_k X_{n-k}=   \varepsilon_n,\,n\in\Z$' in such a way that its solution $\left(X_n\right)$ is integer valued. This is achieved by giving the error terms a distribution supported on $\N_0$ and substituting all multiplications with independent $\N_0$-valued operations. The following notation borrowed from \citet{steutel79} makes the analogy particularly obvious.
\begin{definition}\label{reproduction_operator}
For an $\N_0$-valued random variable $Y$ and a constant $\alpha\geq 0$ define the \emph{reproduction operator} $\circ$ by
$$
\alpha \circ Y:=\sum\limits_{k=1}^Y\xi^{(\alpha)}_k,
$$ where $\xi^{(\alpha)}_1, \xi^{(\alpha)}_2 ,\dots $ are i.i.d.\ and independent of $Y$ with $\xi^{(\alpha)}_1\sim \mathrm{Poisson}(\alpha)$. We use the convention that $\sum_{k=1}^0\xi^{(\alpha)}_k = 0$. 
\end{definition}
We immediately present the multivariate version of the reproduction operator and the 
 multivariate version of the INAR($p$):
 \begin{definition}\label{mv_reproduction_operator}
For a $d\times d$ matrix $A=\left(\alpha_{i,j}\right)_{1\leq i,j\leq d} \in\R_{\geq 0}^{d\times d}$ and an $\mathbb{N}^d$-valued random variable $\bfX = \left(X_1,X_2,\dots,X_d\right)^\top$, define the \emph{multivariate reproduction operator} $\circledast$ by
$$
A\circledast \bfX := \left(\begin{array}{c}\sum\limits_{j=1}^d \alpha_{1,j}\circ X_j \\\dots   \\ \sum\limits_{j=1}^d \alpha_{d,j}\circ X_j  \end{array} \right),
$$
where the reproductions $(\alpha_{i,j}\circ \cdot \,)$ 
operate independently over $1\leq i,j\leq d$.
\end{definition}

\begin{definition} \label{mvINAR(p)}
Let $d, p \in \N$, $A_k \in \R_{\geq 0}^{d\times d}$, $k=1,\dots,p,$ $\bfa_0\in  \R_{\geq 0}^{d}$, and $\left(\bfepsilon_n\right)_{n\in\Z}$ an i.i.d.\ sequence of vectors in $\N_0^d$ with mutually independent components $\bfepsilon_{0,i}\sim \operatorname{Pois}(\bfa_{0,i})$, $i= 1,\dots ,d$. A $d$\emph{-variate INAR($p$) sequence} is a stationary sequence $\left(\bfX_n\right)_{n\in\Z}$ of $\N_0^d$-valued random vectors; it is a solution to the system of stochastic difference-equations
$$
\bfX_n=\sum\limits_{k=1}^pA_k\circledast \bfX_{n-k} + \bfepsilon_n,\quad n\in\Z,
$$ 
where the `\,$\circledast$' operate independently over $k$ and $n$ and also independently of $\left(\bfepsilon_n\right)$. We refer to $\bfa_0$ as \emph{immigration-parameter vector} and to $A_k,\, k=1,2,\dots,p,$ as \emph{reproduction-coefficient matrices}.
\end{definition}
This model has first been considered in~\citet{latour97}. In the same paper we find that
if all zeros of
\begin{align}
z \mapsto \det\left(z1_{d\times d}-\sum\limits_{k=1}^pA_k\right), \quad z\in\mathbb{C}, \label{INAR_stability_criterion}
\end{align}
lie inside the unit circle, then a multivariate INAR($p$) process as in Definition~\ref{mvINAR(p)} exists. 

Consider a univariate INAR($p$) sequence $(X_n)$ with immigration parameter $\alpha_0$ and reproduction coefficients $\alpha_k,\, k = 1,\dots,p$.  Note that the criterion from above now simply reads $\sum_{k=1}^p\alpha_k <1$. Under this condition, we have that
$ 
X_{n} |X_{n-1},X_{n-2} ,\dots\strut\sim \text{Pois}\left( \alpha_0 + \sum_{k=1}^p \alpha_k X_{n-k}\right)$. In particular, 
$
\E[X_n|\sigma( X_{n-1}, X_{n-2},\dots)] =\alpha_0 +  \sum_{k=1}^p \alpha_k X_{n-k}
$---which is the exact version of \eqref{approx_bin_count_eq}. 
The INAR($p$) sequence has a similar immigrant--offspring structure as the Hawkes process. In the time series case, the (possibly multiple) immigrants at each time step stem from i.i.d.\ Pois($\alpha_0$) variables. Each of these immigrants produces Pois($\alpha_k$) new offspring events at $k$ time steps later. Each of these offspring events again serves as parent event for new offspring etc. \\
\par A more obvious choice for the distribution of the counting sequences in Definition \ref{reproduction_operator} would be Bernoulli---yielding the original~\emph{thinning operation} from \citet{steutel79}. Note, however, that
for small reproduction coefficients, the Poisson and the Bernoulli approaches are very similar. Also note that the Poisson distribution is more convenient for our purpose: 
we want to interpret the INAR($p$) model as an approximation of the bin-count sequence of a Hawkes process and in the Hawkes model, an event can have potentially more than one direct offspring event in a future time-interval. In addition, in the Poisson case, we do not have to exclude reproduction coefficients larger than one. 
%
 

\subsection{Approximation of the Hawkes process by the INAR($p$) model} \label{connection}
We examine the close relation between Hawkes point processes and INAR time series in \citet{kirchner16c}. 
For a particularly obvious parallel, the reader may consider the analogy of the existence criteria \eqref{stability_criterion} and \eqref{INAR_stability_criterion}.
Our cited paper gives a precise convergence statement for the univariate case. After establishing existence and uniqueness of the INAR($\infty$) process as a generalization of Definition~\ref{mvINAR(p)} with $d=1$ and $p=\infty$, we prove
\begin{theorem}\label{convergence}
Let $N$ be a univariate Hawkes process with baseline intensity $\eta>0$ and piecewise-continuous excitement function $h:\R_{\geq 0}\to \R_{\geq 0}$ such that $\sum_{k =1}^\infty h\left(k\Delta\right)\Delta < 1$ for all $\Delta \in (0,1)$. Furthermore, let $(X^{(\Delta)}_n)$ be a univariate INAR($\infty$) sequence with immigration parameter $\alpha^{(\Delta)}_0:=\Delta \eta$ and reproduction coefficients 
$\alpha^{(\Delta)}_k := \Delta h\left(k\Delta\right), \,k\in\N$,
and define a family of point processes by
$$
N^{(\Delta)}\big((a,b]\big):= \sum_{n:\, n\Delta \in (a, b]}X^{(\Delta)}_n,\quad a < b,\, \Delta\in (0,1).
$$
Then we have that, for $\Delta\downarrow 0$, the INAR($\infty$)-based family of point processes
$
 \left(N^{(\Delta)}\right)$ converges weakly to the Hawkes process $N$.
\end{theorem}
\begin{proof}
This is Theorem 2 in \citet{kirchner16c}.
\end{proof}
Note that weak convergence of point processes is equivalent to convergence of the corresponding finite-dimensional distributions; see \citet{daley03}, Theorem 11.1.VII. 
The other theoretical result that is important for our estimation purpose is the fact that INAR($\infty$) processes can be approximated by 
INAR($p$) processes, $p<\infty$:
\begin{proposition}\label{INAR(p)_approximation}
Let $\left(X_n\right)$ be an INAR($\infty$) sequence with immigration parameter $\alpha_0>0$ and reproduction coefficients $\alpha_k\geq0,\, k\in\N$. Furthermore, let $\big(X^{(p)}_n\big)$ be a corresponding INAR($p$) sequence, where the reproduction coefficients are truncated after the $p$-th lag. That is, $\big(X^{(p)}_n\big)$ has immigration parameter $\alpha_0^{(p)}:=\alpha_0$ and reproduction coefficients $\alpha^{(p)}_k:= 1_{\{k\leq p\}}\alpha_k,\; k \in\N.$
Then, for $p\to\infty$, the finite-dimensional distributions of $\big(X^{(p)}_n\big)$ converge to the finite-dimensional distributions of $\big(X^{}_n\big)$.
\end{proposition}
\begin{proof}
{This can be derived by establishing the convergence of the corresponding moment-generating functions from Proposition 2 in \citet{kirchner16c}.}
\end{proof}
We have not worked out the multivariate versions of Theorem~\ref{convergence} and Proposition~\ref{INAR(p)_approximation} above. However, the simulations presented further down in the paper support the assumption that both results also hold in the multivariate case. 
Under this assumption, we have the following approximation:
\begin{framed}
\label{approximation}
\emph{\centerline{Basic approximation}}\strut
\strut\\
Let $\bfN$ be a $d$-variate Hawkes with baseline-intensity vector $\bfeta$ and excitement function $H$ as in \eqref{multivariate_intensity}. Let $\big(\bfX_n^{(\Delta)}\big)$ be a $d$-variate INAR($\infty$) sequence with immigration-parameter vector $\bfa_0^{(\Delta)}:= \Delta\bfeta$ and reproduction-coefficient matrices  $A_k^{(\Delta)} := \Delta H(k\Delta),\, k\in\N$. Furthermore, for $p\in\N$, let $\big(\bfX_n^{(\Delta,p)}\big)$ be a corresponding INAR($p$) sequence with baseline intensity $\bfa^{(\Delta,p)}_0:=\bfa_0^{(\Delta)} $ and $p$ reproduction-coefficient matrices $A_k^{(\Delta,p)} :=A_k^{(\Delta)},\,k = 1,\dots, p$. Then, for small $\Delta>0$ and large $p\Delta>0$, we have that 
 \begin{align*}
&\Big(\emph{N}\big((0,\Delta]\big),\emph{N}\big((\Delta ,2\Delta]\big),\dots, \emph{N}\big(((m-1)\Delta ,m\Delta] \big)\Big) \\
& \stackrel{d}{\approx}  \Big(\emph{X}^{(\Delta)}_1,\emph{X}^{(\Delta)}_2,\dots, \emph{X}^{(\Delta)}_m\Big)\\
& \stackrel{d}{\approx}  \Big(\emph{X}^{(\Delta,p)}_1,\emph{X}^{(\Delta,p)}_2,\dots, \emph{X}^{(\Delta,p)}_m\Big),\quad m\in\N. 
\end{align*}
If $\text{supp}(H)\subset [0,s]$ for some finite $s>0$, then the second approximation becomes an equality for all $p\geq \lceil s/\Delta\rceil$.
\end{framed}
The approximation summarized in the box above is the key observation for our \emph{estimation procedure}: 
\begin{itemize}
\item[(i)] Choose a small bin-size $\Delta>0$ and calculate the bin-count sequence of the events stemming from the Hawkes process.
\item[(ii)] Choose a large support $s:=p\Delta$ and fit the approximating INAR($p$) model to the bin-count sequence via conditional least-squares.
\item[(iii)]  Interpret the scaled immigration-parameter estimate $\hat{\bfa}_0^{(\Delta,p)}/\Delta$ as the natural candidate for an estimate of $\bfeta$ and, for $k\in\{1,2,\dots,p\}$, interpret
 the scaled reproduction-coefficient matrix estimates  $\hat{A}^{(\Delta,p)}_k/\Delta$ as natural candidates for estimates of $ H\left(k\Delta\right)$. 
\end{itemize}
Before giving the formal definition of the estimator in the next section, we illustrate the power of the presented method in Figure~\ref{fig2}. 
\section{The estimator}   \label{The Estimator}
In this section, we first discuss estimation of the approximating INAR($p$) process. Then we define our Hawkes estimator formally and collect some of its properties. {We describe alternative estimation methods and compare them with our estimator.} Furthermore, we present results of a multivariate simulation study that support our approach. 

\subsection{Estimation of the INAR($p$) model}\label{CLS_INAR}
There are several possibilities to estimate the parameters of an INAR($p$) process.
As the margins are conditionally Poisson distributed, in principle, maximum-likelihood estimation (MLE) can be applied. In our context, however, numerical optimization of the likelihood is difficult, as the number of model parameters will typically be 
very large. A method-of-moments type estimator would be the Yule--Walker method (YW). 
A third method is the conditional least-squares estimation (CLS). 
We formulate the estimation in terms of CLS; {see Section~\ref{alternative_estimation_methods} for this choice.} CLS-estimation in the univariate INAR($p$) context has been discussed, e.g., in \citet{li91} and \citet{zhang10}. In both papers, the reasoning is performed along the lines of \citet{klimko78}, which was originally developed for CLS-estimation of time series with the very general structure 
`$
\E\left[X_n|X_{n-1},\dots \right]=g_\theta(X_{n-1},X_{n-2},\dots )
$',
where $g_\theta$ may be nonlinear.
However, as already noticed in \citet{latour97}, INAR($p$) sequences can be represented as standard AR($p$) models with white noise immigration terms. This yields ways for inference that are more direct.
 
\begin{proposition} \label{white_noise}
Let $\left(\bfX_n\right)$ be a $d$-dimensional INAR($p$) sequence as in Definition~\ref{mvINAR(p)}  with immigration-parameter vector $\bfa_0\in\R_{\geq 0}^d\setminus \{0_d\}$ and reproduction-coefficient matrices $A_k\in\R_{\geq 0}^{d\times d},\,k = 1,2,\dots,p$, such that \eqref{INAR_stability_criterion} holds.  Then 
$$
\bfu_n:= \bfX_n - \bfa_0 - \sum\limits_{k=1}^p A_k\bfX_{n-k},\quad n\in\Z ,
$$
defines a (dependent) white-noise sequence, i.e.,
$\left(\bfu_n\right)$ is stationary, $\E \bfu_n= 0_d,\ n\in\Z$, and 
$$
\E\left[\bfu_n\bfu_{n'}^\top\right]=\begin{cases}  \diag\left(\left(1_{d\times d} -  \sum\limits_{k=1}^pA_k\right)^{-1}\right)
 ,&n=  n',\\
0_{d\times d},& n \neq n'.
\end{cases}
$$
\end{proposition}
\begin{proof}
This can be shown by straightforward (if lengthy) calculations; see Appendix~\ref{pf_white_noise}.
\end{proof}
As a consequence of Proposition~\ref{white_noise}, a $d$-variate INAR($p$) process can be represented as a standard $d$-variate autoregressive time series with (dependent) white-noise errors:
\begin{corollary}\label{AR-representation} Let $\left(\bfX_n\right)$ be the multivariate INAR($p$) sequence and $(\bfu_n)$ the white-noise sequence from Proposition~\ref{white_noise}. Then $\left(\bfX_n\right)$ solves
the system of stochastic difference-equations
$$
\bfX_n = \mathbf{a}_0 + \sum\limits_{k=1}^p A_k \bfX_{n-k} + \bfu_n,\quad n\in\Z.
$$
\end{corollary}
Such vector-valued time series with linear autoregressive structure 
have early on been examined; see, e.g., \citet{hannan70}. However, estimation in a multivariate context requires cumbersome notation. 
In order to make our results comparable, 
we follow one reference throughout, namely the monograph \citet{luetkepohl05}. Adapting its notation is also the reason why we work with wide matrices---i.e., matrices having a number of columns in the order of the sample size---instead of the more common long matrices.
\begin{definition}\label{CLS}
Let $(\mathbf{x}_k)_{k\in\N}$ be an $\R^d$-valued sequence, where we interpret $\bfx_k$ as a column vector. Fix $p$ and $n\in\N$, $p < n$, and define the \emph{multivariate conditional least-squares estimator} as
\begin{eqnarray*}
\hat{\theta}^{(p,n)}_{CLS}: \R^{d\times n} & \longrightarrow &\R^{d\times (dp+1)}\\
(\mathbf{x}_1,\dots ,\mathbf{x}_n)&\longmapsto& \hat{\theta}^{(p,n)}_{CLS}\left(\mathbf{x}_1,\dots ,\mathbf{x}_n\right):=  \bfY\bfZ^\top\left(\bfZ \bfZ^\top\right)^{-1},
\end{eqnarray*}
where
$$
\bfZ\left(\mathbf{x}_1,\dots ,\mathbf{x}_n\right):=\left(\begin{array}{cccc}\mathbf{x}_{p} & \mathbf{x}_{p+1}& \dots   & \mathbf{x}_{n-1} \\
\mathbf{x}_{p-1} & \mathbf{x}_{p} & \dots  & \mathbf{x}_{n-2} \\\dots  & \dots  & \dots  & \dots  
 \\
\mathbf{x}_{1} & \mathbf{x}_{2} & \dots    & \mathbf{x}_{n-p} 
\\ 1 & 1 & \dots & 1\end{array}\right)\in\R^{(dp+1)\times (n-p)}
$$
is the \emph{design matrix}
and 
$
\bfY\left(\mathbf{x}_1,\dots ,\mathbf{x}_n\right):=\left(\mathbf{x}_{p+1}, \mathbf{x}_{p+2}, \dots  ,\mathbf{x}_{n}\right)\in\R^{d\times(n-p)}.
$
\end{definition}
\noindent Dealing with multivariate time series the following notations turn out to be useful:
\begin{definition}
The  $vec(\cdot)$\emph{-operator} takes a matrix as its argument and stacks its columns. The binary $\otimes$-operator is the \emph{Kronecker operator}: for an $m\times n$ matrix $A = (a_{i,j})$ and a $p\times q$ matrix $B$, $(A\otimes B)$ is the $mp\times nq$ matrix consisting of the block-matrices $a_{i,j} B$, $i= 1,\dots, m,\; j = 1,\dots,n$.
\end{definition}
The $\vec$-notation arises because the estimator is matrix-valued and we have no notion of the covariance of a random matrix. As we will see the $\otimes$-notation is strongly related to the $\vec$-operator.
For a large collection of properties of these operators; see Appendix A of \citet{luetkepohl05}. 
The following theorem collects all relevant information for CLS-estimation of multivariate INAR($p$) sequences. Together with the approximation results from Section~\ref{approximation}, this theorem is the theoretical basis for our Hawkes estimation procedure.

\begin{theorem}\label{inference}
Let $\left(\bfX_n\right)$ be a $d$-dimensional INAR($p$) sequence as in Definition~\ref{mvINAR(p)}  with immigration-parameter (column) vector $\bfa_0\in\R_{\geq 0}^d\setminus \{0_d\}$, and reproduction-coefficient matrices $A_k\in\R_{\geq 0}^{d\times d},\,k\in\{1,2,\dots,p\}$, such that $\mathrm{spr}\left(\sum_{k=1}^p A_k\right) < 1$. 
 Let
\begin{align*}
\bfB&:=\big(A_1, A_2,\dots  ,A_p, \bfa_0\big) \in\R^{d\times (dp+1)}
\quad  \text{and }\\
 \hat{\bfB}^{(n)}&:=\hat{\theta}_{CLS}^{(p,n)}\big((\bfX_k)_{k=1,\dots ,n}\big)\in\R^{d\times (dp+1)}
\end{align*}
the CLS-estimator with respect to the sample $(\bfX_k)_{k=1,\dots ,n}$. Then $\hat{\bfB}^{(n)}$ is a weakly consistent estimator for $\bfB$. Furthermore, let $\bfZ$ be the design matrix from Definition~\ref{CLS} with respect to $(\bfX_k)_{k=1,\dots ,n}$. Assume that the limit 
\begin{align}
\frac{1}{n-p}\bfZ\bfZ^\top\stackrel{\mathrm{p}}{\longrightarrow}:\Gamma\in\R^{{(dp+1)\times (dp+1)}},\quad n\longrightarrow \infty,\label{Gamma}
\end{align}
exists and is invertible. In addition, assume that the model is irreducible in the sense that $\P[\bfX_{0,i} = 0]<1,\, i = 1,2,\dots,d$. Then, for the asymptotic distribution of $\vec\left(\hat{\bfB}^{(n)}\right)\in\R^{d^2p+d}$, one has, for $n\to\infty$,
\begin{align*}
\sqrt{n-p}\Big( \vec\big(\hat{\bfB}^{(n)}\big)&- \vec\big(\bfB\big) \Big)\\
&\stackrel{\mathrm{d}}{\longrightarrow}\;
\mathcal{N}_{d^2p+d}
\Big( 0_{d^2p+d},\left(\Gamma^{-1}\otimes  1_{d\times d}\right)W \left(\Gamma^{-1}\otimes  1_{d\times d}\right)\Big), 
\end{align*}
where
\begin{align}
W:=
\E\left[\Big(\bfZ_0\otimes 1_{d \times d}\Big)
\bfu_0\Big(\big(\bfZ_0\otimes 1_{d \times d}\big)\bfu_0\Big)^\top \right]
\label{W}\in\R^{(d^2p+d)\times(d^2p+d)}
\end{align}
with
$$
\bfu_0:= \bfX_0 - \mathbf{a}_0 - \sum\limits_{k=1}^pA_k \bfX_{-k}
\quad
\text{and}
\quad
\bfZ_0 := \Big(\bfX^\top_{-1}, \bfX^\top_{-2},\dots ,\bfX^\top_{-p},1\Big)^\top.
$$
\end{theorem}
\begin{proof}
In view of 
Corollary~\ref{AR-representation},
 it suffices to prove Theorem~\ref{inference} for the corresponding vector-valued autoregressive time series. So the distributional properties of the CLS-estimator can be derived similarly as in \citet{luetkepohl05}, pages 70--75, where independent errors are assumed. We provide a highly self-contained proof for the dependent white-noise case in Appendix~\ref{pf_inference}.
\end{proof}
Note that the condition $\P[\bfX_{0,i} = 0]<1,\, i = 1,2,\dots,d,$ in Theorem~\ref{inference} above is purely technical: if we had $\P[\bfX_{0,i_0} = 0]=1$ for some $i_0$, this would imply that in one component of our sample we cannot observe any events. We may exclude this case with a clear conscience.

\subsection{The Hawkes estimator}\label{hawkes_estimator}
Combining Theorem~\ref{inference} with the basic approximation from Section~\ref{connection} yields the following estimator for multivariate Hawkes processes:
\begin{definition} \label{estimator}
Let $\bfN= \left(N^{(1)},N^{(2)},\dots,N^{(d)}\right)$ be a $d$-variate Hawkes process with baseline-intensity vector $\bfeta\in\R_{\geq 0}^d\setminus \{0_d\}$ and excitement function $H=(h_{i,j}):\R_{\geq 0}\to \R_{\geq 0}^{d\times d}$ such that \eqref{stability_criterion} holds. 
Let $T>0$ and
consider a sample of the process on the time interval $(0,T]$.
For some $\Delta>0$, construct the $\N^d_0$-valued \emph{bin-count sequence} from this sample:
\begin{align}
\bfX^{(\Delta)}_k := \bigg(N^{(j)}\Big(\big((k-1)\Delta,k\Delta\big]\Big)\bigg)_{j = 1,\dots ,d},\quad k = 1, 2, \dots , n := \left\lfloor T/\Delta\right\rfloor.\label{bin_count_sequence}
\end{align}

Define the \emph{multivariate Hawkes estimator} with respect to some support $s,\, \Delta < s <T$, by applying the CLS-operator from Definition~\ref{CLS} with maximal lag $p:=\lceil s/\Delta\rceil$ on these bin-counts:
\begin{align}
\hat{\bf{H}}^{{(\Delta,s)}}:=  \frac{1}{\Delta} \hat{\theta}^{(p,n)}_{CLS}\left(\left(\bfX^{(\Delta)}_k\right)_{k=1,\dots ,n}\right).\label{estimator_calculation}
\end{align}
\end{definition}
\noindent We collect the main properties of the estimator in the following remark.
\begin{remark} \label{cov}
The following additional notation clarifies what the entries of the $\hat{\bf{H}}^{{(\Delta,s)}}$ matrix actually estimate:
\begin{align}
\left(\hat{H}_1^{{(\Delta,s)}}, \dots , \hat{H}_p^{{(\Delta,s)}},\hat{\bfeta}^{{(\Delta,s)}}\right):=\hat{\bf{H}}
^{{(\Delta,s)}}\label{notation}
\end{align}
From Theorem~\ref{inference} on estimation of INAR($p$) sequences together with the basic approximation in Section~\ref{connection}, we see that, for $0<t<s$,
$$
\left(\hat{H}_{\lfloor t/ \Delta\rfloor}^{{(\Delta,s)}}\right)_{ij},\quad  i,j = 1,\dots ,d ,\quad \text{ respectively, }\quad \left(\hat{\bfeta}^{(\Delta,s)}\right)_{i},\quad i= 1,\dots ,d, 
$$
are weakly consistent estimates (for $T\rightarrow \infty$, $\Delta\rightarrow 0$ and $s = \Delta p\rightarrow \infty)$
for the excitement-function component value $h_{i,j}(t)$, respectively, for the baseline-intensity vector component $\bfeta_i$. Furthermore, we find from Theorem~\ref{inference} that
\begin{align}
\vec\left(\hat{\mathbf{H}}^{{(\Delta,s)}}\right) \stackrel{\mathrm{approx.}}{\sim}\mathcal{N}_{d^2p+d}\left(\vec\left(\mathbf{H}\right),S^2\right),\label{normal_approximation}
\end{align}
with
$$
S^2 := \frac{1}{\Delta^2 (n-p)} \left(\Gamma^{-1}\otimes  1_{d\times d}\right) W \left(\Gamma^{-1}\otimes  1_{d\times d}\right),
$$
where $\Gamma$ and $W$ are defined as in \eqref{Gamma} and \eqref{W} with respect to the bin-count sequences. Substituting $\Gamma$ and $W$ with their empirical versions yields the covariance estimate

\begin{align}
\hat{S}^2:=\frac{1}{\Delta^2}\left(\left(\bfZ\bfZ^\top\right)^{-1}  \otimes 1_{d\times d}\right)\sum\limits_{k={p+1}}^n{\bfw_k \bfw_k^\top} \left(\left(\bfZ\bfZ^\top\right)^{-1}  \otimes 1_{d\times d}\right),\label{cov_estimator}
\end{align}
where $\bfZ$ is the design matrix from Definition~\ref{CLS} with respect to the bin-count sequence and, for $k= p+1, p+2, \dots, n$,
\begin{align*}
\bfw_k
&:= \left(\left(\left(\bfX^{(\Delta)}_{k-1}\right)^\top, \left(\bfX^{(\Delta)}_{k-2}\right)^\top,\dots ,\left(\bfX^{(\Delta)}_{k-p}\right)^\top,1\right)^\top\otimes 1_{d\times d}\right)\\
&\hspace{4cm}\cdot\,\left(\bfX_k^{(\Delta)} - \Delta\hat{\bfeta}^{(\Delta,s)} - \sum\limits_{l=1}^p\Delta\hat{H}_l^{{(\Delta,s)}} \bfX_{k-l}^{(\Delta)}\right).
\end{align*}

Following formulas 
are useful for implementation of confidence intervals:
 \begin{align}
{\Cov}\left(
\left(
\hat{H}_{k_1}^{(\Delta,s)}
\right)_{i_1,j_1}, 
\left(\hat{H}_{k_2}^{(\Delta,s)}
\right)_{i_2,j_2}
\right) 
& 
= 
{S}^2_{(k_1 - 1) d^2 + (j_1 - 1) d + i_1,\, (k_2 - 1) d^2 + (j_2 - 
1) d + i_2}, \label{cov_est_calc}
\end{align}
for $i_1, i_2, j_1, j_2 \in\{1,\dots, d\}$ and $k_1,k_2\in\{1,\dots, p\}$.
\begin{align*}
\Cov\left(
\hat{\bfeta}_{i_1}^{(\Delta,s)}
, 
\hat{\bfeta}_{i_2}^{(\Delta,s)}\right)& =
S^2_{p d^2 + i_1,\, p d^2 p+ i_2},\quad \text{for }i_1,i_2\in\{1,...,d\}.
\end{align*}
\end{remark}
\noindent Applying Remark~\ref{cov} above together with Definitions~\ref{CLS} and~\ref{estimator}, our Hawkes estimation procedure may be implemented in a straightforward manner.
However, we emphasize that the resulting matrix $\hat{\mathbf{H}}^{{(\Delta,s)}}$ in \eqref{estimator_calculation} does not completely specify a fitted Hawkes model; it only yields pointwise estimates on a grid, whereas the true excitement-parameter is a function on $\R_{\geq 0}$; see Section~\ref{Hawkes_Process}. To complete the estimation, we have to apply some kind of smoothing method over the pointwise estimated values.
We work with cubic splines, normal kernel smoothers, and local polynomial regression ({\tt ksmooth()}, 
{\tt smooth.spline()} and
{\tt loess()} in {\tt R}). We find that the results do not vary significantly. The choice of the estimation parameters bin-size $\Delta$ and support $s$ has more impact. Therefore, we focus on the selection of these estimation parameters; see Section~\ref{refinements} The smoothing idea will be relevant in Section~\ref{choice_of_bin_size}, where we discuss variance issues.
In many applications, one can even avoid choosing and applying a smoothing method: practitioners might want to use our estimation procedure from Definition~\ref{estimator} for identifying or rejecting certain parametric models. For such purposes, the pointwise estimates suffice. The same is true if the estimation procedure is used as a mere tool for representing large event data sets; see Section~\ref{interpretation}. Finally, one is often only interested in the integral of the excitement; see the comments after~\eqref{stability_criterion}. In this case, it makes more sense to directly add up the estimates rather than to take the detour over some smoothing method.

\subsection{Simulation studies}\label{simulation_study}
We check the distributional properties of the Hawkes estimator collected in Remark~\ref{cov} in a first simulation study. The results are summarized in Figures~\ref{fig15},{~\ref{fig6}, and {\ref{rev1}}. {Note that at this point we omit the question of selection methods for the estimation parameters $s$ and $\Delta$ .}  This issue will be discussed separately in Sections~\ref{choice_of_support} and~\ref{choice_of_bin_size}. {Already in Figure~\ref{rev1}, however,  the impact of this choice on the estimation results is illustrated.}
\begin{figure}
\begin{center}
\includegraphics[width = 0.32\textwidth]{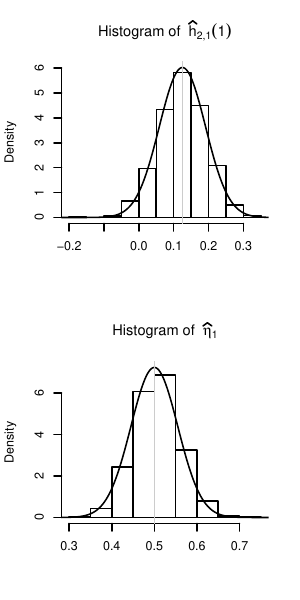}\hfill
\includegraphics[width =0.32 \textwidth]{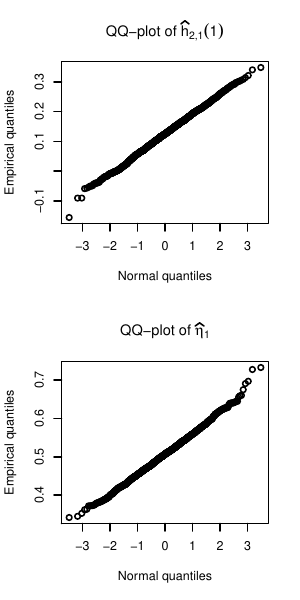}\hfill
\includegraphics[width =0.32\textwidth]{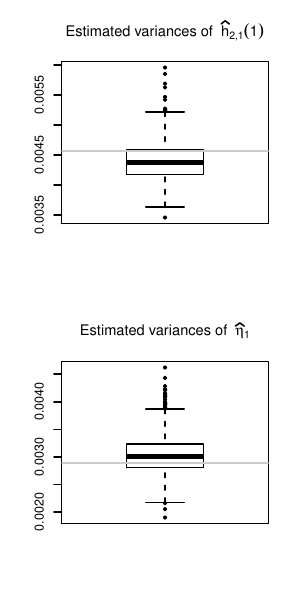}
\caption{Illustration of the simulation study described in Section~\ref{simulation_study}. The study confirms the distributional properties of our estimation procedure collected in Remark~\ref{cov}: We simulate 2\,000 times from the bivariate Hawkes process introduced in Figure~\ref{fig1}. In each simulation, we realize about 5\,000 events in each component. For all of these samples, we calculate our estimator from Definition~\ref{estimator} as well as the covariance estimator from \eqref{cov_estimator}. These calculations depend on two parameters, the support $s$ and the bin-size $\Delta$. We apply $s=6$ together with a relatively coarse bin-size $\Delta =0.2$. The upper-row panels illustrate the estimation of $h_{2,1}(1)=0.5 (1+1)^{-2} = 0.125$; the lower-row panels illustrate the estimation of the baseline-intensity component $\eta_1 = 0.5$. \emph{Left column panels:} the asymptotic normal densities around the true values (grey vertical lines) are added to the histograms. The grey vertical lines refer to the true values. The means of the estimates (not illustrated) would cover the true values. \emph{Middle column panels:} the QQ-plots support the asymptotic normality result. \emph{Right column panels:} the boxplots collect the 2\,000 estimated variances; see~\eqref{cov_estimator}. The horizontal grey lines refer to the empirical variance of the 2\,000 estimates.}
\label{fig15}
\end{center}
\end{figure}
\subsubsection*{{Bivariate estimation}}
We simulate $2\,000$ times from a bivariate Hawkes model with baseline intensity $\bfeta = \left(\eta_1,\eta_2\right)^\top = (0.5,0.25)^\top$ and excitement function
\begin{align}
H(t)&=\left(\begin{array}{cc} h_{1,1}(t) & h_{1,2}(t)\\
h_{2,1}(t) & h_{2,2}(t)
\end{array}\right) 
= \left(\begin{array}{cc}0& 1_{1<t\leq 3} 0.25\\
0.5(1+t)^{-2}\quad & 1_{t\leq \pi} 0.2 \sin(t)
\end{array}\right) ;\label{bivariate_example}
\end{align}
see Figure~\ref{fig1} for this parametrization and Figure~\ref{fig2} for an estimation of a single realization. In each simulation, about 5\,000 events in each component are generated and our Hawkes estimator \eqref{estimator_calculation} is calculated. 
We apply a bin size $\Delta = 0.2$ and a support parameter $s=6$. 
These calculations yield 2\,000 matrices of the form $\hat{\mathbf{H}}^{{(\Delta,s)}}\in\R^{2\times121}$. We examine the estimations of  $\eta_1=0.5 $, i.e., the baseline-intensity for the first component, and the estimations of  $h_{2,1}(1)=
 0.125$, i.e., the crossexcitement on component 2 from component 1 after one time unit. These values correspond to the entries $\hat{\mathbf{H}}^{{(\Delta,s)}}_{1,121}$ and $\hat{\mathbf{H}}^{{(\Delta,s)}}_{2,9}$ in the estimator matrices. We find that the 2\,000 estimates are distributed symmetrically around the true values. The means of the estimates correspond almost completely to the true values. QQ-plots support the asymptotic normality result. For both estimations, we also calculate the variance estimates from \eqref{cov_estimator}. Comparing the empirical variance of the 2\,000 estimates with the 2\,000 estimated variances confirms the analytic result. Furthermore, the empirical covering rates for the 95\%-confidence intervals are 94.5\% for the baseline-intensity estimate, respectively, 94.8\% for the excitement-value estimate. Note that the applied estimation parameters $\Delta = 0.2$ and $s = 6$ are considerably `wrong': the bin-size is quite large and the true support of $H$ would be $\infty$. We may interpret the successful estimation as a sign for the robustness of the method with respect to the estimation parameters.
\subsubsection*{{Variance of the estimates}} \begin{center}
\begin{figure}
\includegraphics[width = \textwidth]{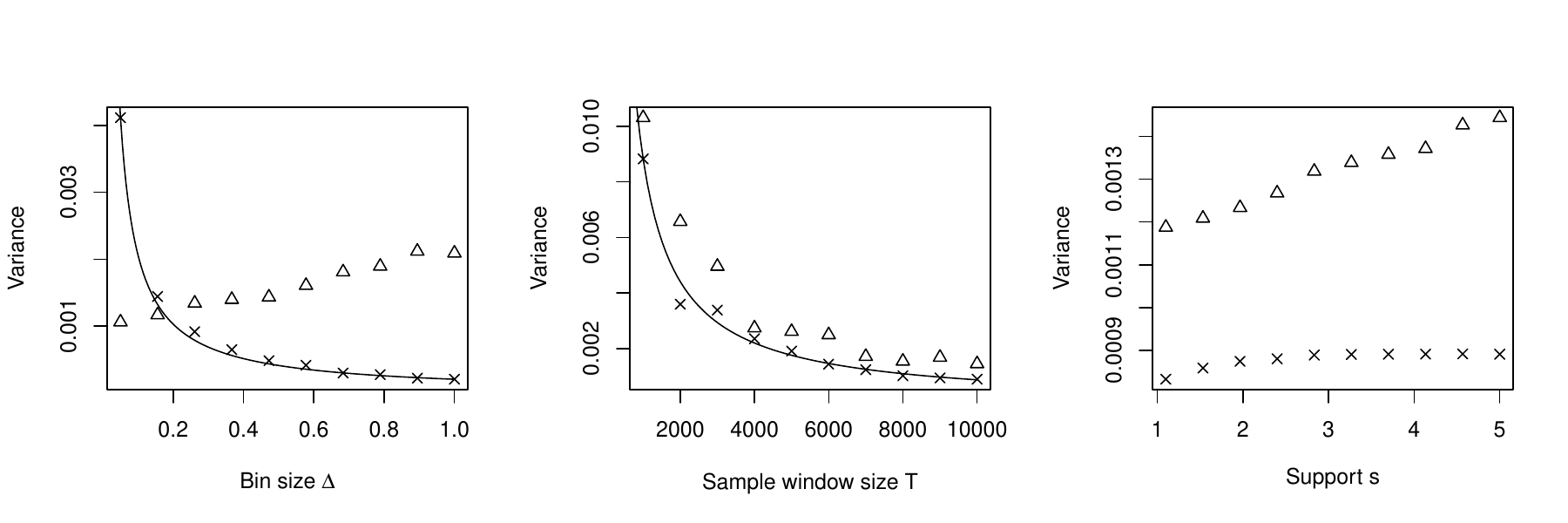}
\caption{The Hawkes estimator from Definition~\ref{estimator} depends on the bin size $\Delta$, on the size of the sample window $T$ and on the support parameter $s$. We examine empirically how the variances of the estimates depend on these three parameters. We simulate a very large sample from a univariate Hawkes process with excitement function $h: t\mapsto 1_{t\leq 3}(1+t)^{-2}$ and baseline intensity $\eta=1$. With respect to this single sample, we calculate the estimated variance for the estimates of $h(1) = 0.25$ (crosses) and $\eta = 1$ (triangles) using different $\Delta$,  $T$  and $s$; see \eqref{cov_estimator}. The solid lines in the two left panels are $\Delta \mapsto c_1\Delta^{-1}$, respectively, $T \mapsto c_2 T^{-1}$, for some constants $c_1,c_2 > 0$. The curves fit the variance estimates of the excitement-function estimate well. In contrast, the variance of the baseline estimate (triangles) is relatively constant with respect to $\Delta$. In the right panel, we see that the larger the support parameter $s$, the larger the variances become---this seems natural, as we estimate more parameters with respect to the same sample-size.
}
\label{fig6}
\end{figure}
\end{center}
Separately, we examine the impact of the choice of the bin-size $\Delta$, the support $s$ and the size of the sample window $[0,T]$ on the variances of the estimates; see Figure~\ref{fig6}. For various $\Delta$, $s$ and $T$, we calculate \eqref{cov_estimator}, the estimated covariance of the estimator matrix with respect to a single very large sample {of a univariate Hawkes process}. We find that the excitation and baseline estimation variances with respect to sample windows $[0,T]$ are proportional to $T^{-1}$. Variances slightly increase if we increase the support parameter $s$.  The variance of the baseline intensity estimate with respect to $\Delta$ is roughly constant in $\Delta$. However, the variance of the excitement estimate with respect to $\Delta$ is proportional to $\Delta^{-1}$. Albeit this relation, we will see in Section \ref{choice_of_bin_size} that the excitement estimates are still meaningful for very small values of $\Delta$.

\subsubsection*{{Estimation of the branching coefficient}}
\begin{figure}
\begin{center}
\subfigure[The boxplots summarize the estimates. The grey horizontal lines refer to the true branching coefficient 0.5. The crosses refer to the Riemann-type sum $\sum_{k = 1}^{\lceil s/ \Delta \rceil} \Delta h(k\Delta)$. 
The plots illustrate the three different errors of our method presented in Section~\ref{intuition}. \emph{Cut-off error}; see \eqref{cutoff_error}: for the estimates in the left panel and middle panel, we applied a too small support. This results in an underestimation of the branching coefficient. \emph{Discretization error}; see \eqref{discretization_error}: the boxplot centers catch $\sum_{k = 1}^\infty \Delta h(k\Delta)$ (crosses) rather than $\int h \d t$. \emph{Distributional error} \eqref{distributional_error}: the coarser the value of $\Delta$, the greater the distance from the average estimate to the Riemann-type sums (crosses); see the explanations in Section~\ref{intuition}.]{
\includegraphics[width = \textwidth]{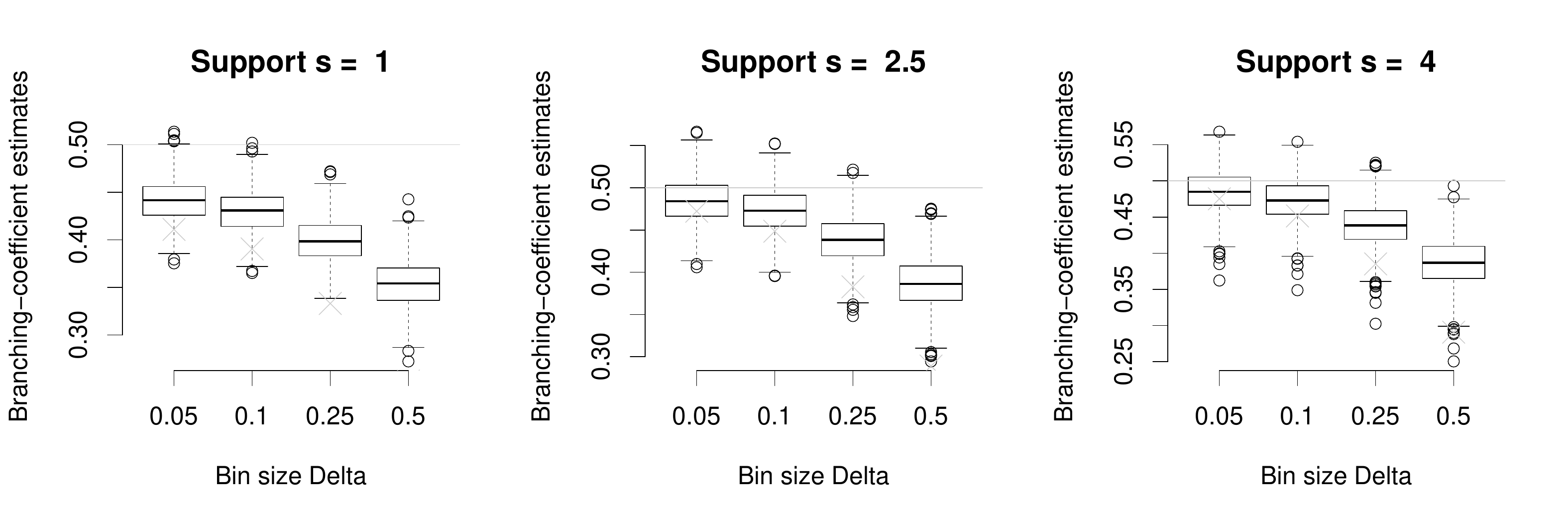}}
\subfigure[The boxplots collect the variance estimates. The crosses on the right of each boxplot refer to the corresponding empirical variance of the 1\,000 estimates together with a bootstrapped 95\%-confidence interval. We furthermore observe that the variance of the branching-coefficient estimates is relatively stable over smaller $\Delta$---this is in contrast with the point-wise estimates, whose variance behaves as $\Delta^{-1}$; see Figure~\ref{fig6}. This means, if the goal of the estimation is a branching-coefficient estimate, there is no reason not to choose $\Delta$ as small as possible.
The variance increases when the support is chosen larger. This is intuitive: we estimate more values with the same amount of data. Furthermore, we observe the center of the variance \emph{estimates} (boxplots) gets nearer and nearer to the empirical variances (crosses) the larger the support becomes. For large $s$ and (relatively) small $\Delta$ the average of the estimates catches the empirical variance perfectely.]{
\includegraphics[width = \textwidth]{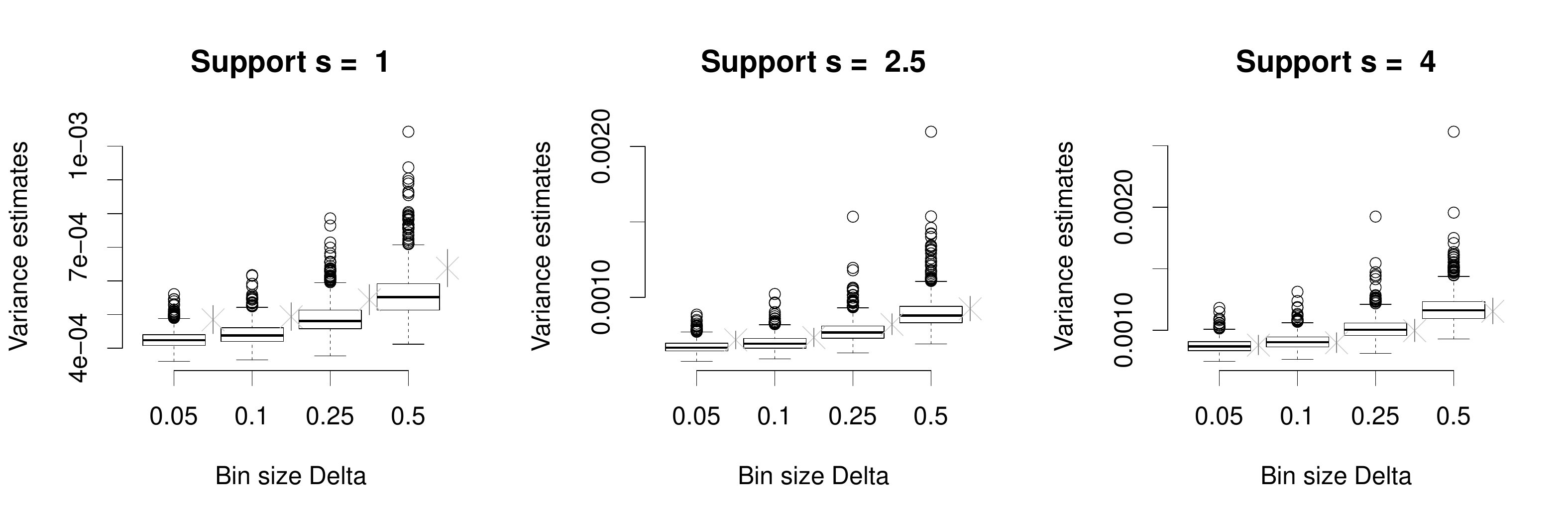}}
\caption{{Estimation of the branching coefficient; see Section~\ref{simulation_study}:  we simulate 1\,000 realizations from a univariate Hawkes process with excitement function $h(t):= \exp(-2t)$. Each realization consists of about 1\,000 events. For each realization we calculate the Hawkes estimator $\hat{H}^{(\Delta, s)}$ from Definition~\ref{estimator} with respect to $\Delta \in \{0.05, 0.1, 0.25, 0.5\}$ and $s\in \{1, 2.5, 4\}$. From each of these estimates, we derive the branching-coefficient estimate $\Delta \sum_{k = 1}^{\lceil s/ \Delta \rceil} H^{(\Delta, s)}_k$. We also calculate the corresponding variance estimates derived from~\eqref{cov_est_calc}. 
}
}\label{rev1}
\end{center}
\end{figure}

 {A particularly meaningful statistic of a Hawkes process is {its} branching matrix{; see the explanations below~\eqref{stability_criterion}}. In the univariate case, this matrix provides the (unconditional) probability that an arbitrarily chosen event has a parent event---that is, that the event is an endogenous event. This explains the alternative notion \emph{rate of endogeneity}. With the notation from \eqref{notation}, we define the natural branching-coefficient estimate $\hat{K}^{(\Delta, s)} := \Delta\sum_{k = 1}^p \hat{H}^{(\Delta, s)}_k$, where $p= \lceil s / \Delta\rceil$. The corresponding variance-estimate can be calculated from \eqref{cov_estimator}, respectively, \eqref{cov_est_calc}. Note that efficient ways to calculate these particular variance-estimates are given in \cite{kirchner16a}. From the same univariate Hawkes model as above, we simulate 1\,000 times and calculate the branching-coefficient estimate as well as the corresponding variance-estimates. We do this with respect to different bin-sizes $\Delta$ and different support-parameters $s$. Figure~\ref{rev1} illustrates the results. In particular, it visualizes the cut-off error, the discretization error, and the distributional error as discussed in Section~\ref{intuition}. It is no surprise that the empirical variance of the estimates also depends on the estimation parameters $\Delta$ and $s$. The variance estimates adapt accordingly.  However, if the support $s$ is chosen (much) too small, then the variance estimates are biased. On the other hand, the bin size $\Delta$ hardly influences the bias of the variance estimator: if the support is chosen large enough, then the variance estimates catch the empirical variance well---even for very coarse bin-sizes. We performed this relatively large simulation study on an ordinary laptop. This did not allow to apply really small bin-sizes. But in smaller simulation studies one finds that the estimation bias becomes essentially invisible for $\Delta \leq 0.01$ (not illustrated). Also note that integrating a smoothed version of the estimates typically captures the true branching-coefficient even better. However, the distribution of this kind of estimates seems less tractable and---again---as long as $\Delta$ is small, the difference between the methods vanishes.
}

{
\subsection{Alternative estimation methods}\label{alternative_estimation_methods}
There are other approaches to Hawkes estimation that propose alternatives to straightforward MLE. We first give an overview. Then we discuss a specific nonparametric method that is closely related to our estimation procedure.
\subsubsection*{Overview}
In \citet{lewis11}, the authors consider the estimation of univariate Hawkes processes with time-varying baseline intensities. Their simulation study shows that the convergence of the algorithm strongly depends on the underlying true model. It is not clear how feasible this approach would be in the multivariate set-up. In \citet{lemonnier14}, the authors approximate the excitement functions as well as the time-varying baseline intensities of a multivariate Hawkes process by sums of exponential functions. They show that the approximating model is a Markov process. This reduces the complexity of the likelihood to linear dependence on the number of observations and allows to apply MLE. In \citet{alfonsi15}, the authors incorporate a bivariate marked Hawkes model into a larger model for the price jumps on a market microstructure level. In this larger model, the Hawkes process models the times when a trade on either side of the market takes place. The authors also consider (i.i.d.) marks that influence the intensity. As in \citet{lemonnier14}, they approximate the excitement functions by sums of exponential functions. This allows to calibrate the model parameters via MLE. In \citet{hansen15}, the authors concentrate on the identification of the nonzero excitement functions of a multivariate Hawkes process by minimizing a least-squares objective of the realized intensity subject to an $l_1$-penalty. The method is comparable to the LASSO method from time series. In \citet{reynaud14}, this method is revisited with respect to diagnostic tests. Note that our method can also deal with the problem of identification of nonzero excitement functions; see the explanations after~\eqref{bm}. This idea is worked out in \citet{kirchner16a}, where we present a simulation study with fairly high-dimensional Hawkes processes ($d = 10$). Our approach has the advantage that the tuning parameter of the estimation has an interpretation in terms of significance of the observed excitement whereas, as a rule, the LASSO method offers no help when it comes to choosing the penalization parameter. 
 
\par Finally, we refer to the alternative nonparametric Hawkes estimation approach first introduced in \citet{bacry11b}. In \citet{bacry14a} the method is further developed and extended for the marked case. \citet{bacry15b} presents an application of the method in the context of high-frequency order flows and price jumps. As this method is closely related to our approach, we discuss it in more detail:

\subsubsection*{Bacry--Muzy method} The starting point of the Bacry--Muzy method (BM-method) is an integral equation of Wiener--Hopf type from \citet{hawkes71a}. This equation relates the autocovariance density of a multivariate Hawkes process with its excitement function. It is the analogue of the Yule--Walker equations for discrete-time processes. The BM-method basically consists of two steps:

\begin{enumerate}

\item  The autocovariance density (respectively, something very related and similar) of the event-stream data is estimated on a grid by a discretization scheme that depends on some bin-size $h$. Smoothing the estimated values yields continuous-time functions. These functions are plugged into a Wiener--Hopf type integral equation as described above.

\item For a maximal support $A>0$ and some $Q\in\N$, the Wiener--Hopf equation is discretized and solved for values of the excitement function by a quadrature with $Q$ quadrature points.
\end{enumerate}

We are not sure if this double discretizing and smoothing---once for the autocovariance density and then for the numerical solution of the Wiener--Hopf equation---is necessary: directly applying Yule--Walker estimation to the bin-count sequences would presumably yield similar results---and less calculation and consideration at that. 
%
%
\par In \citet{bacry15b}, the authors show that their procedure can be extended for the estimation of marked processes as well: under the assumption that the \emph{mark functions} \citep{bacry15b}, respectively, \emph{impact functions} \citep{liniger09} only take a finite number of values, one can express the marked process as a higher dimensional Hawkes process without marks. This latter process is calibrated by the method for the unmarked case, and the estimates are appropriately transformed to the original marked model. In \citet{kirchner16b}, we apply this method in combination with our estimation method.
\par  
As for results, in most cases the estimates of the BM-method and of our method are presumably quite similar if we set $\Delta := h$, $s := A$, and $p := Q := \lceil A / h\rceil$. Or alternatively, if we want to mimic the double smoothing from the BM-method, we set
$\Delta := h$ and $s := A$ as before, and, in addition, $\tau := \lceil A / Q\rceil$. Here, $\tau > \Delta$ denotes a bandwidth over which our excitement estimates from Definition~\ref{estimator} are smoothed; see Section~\ref{choice_of_bin_size}. Note that for the BM-method, the choice of $A$ and the corresponding cut-off error \eqref{cutoff_error} are not discussed whereas  the choice of $s$ will be discussed in Section~\ref{choice_of_support}.  Despite the similarities to the BM-method, we believe that our method offers important new insights into  Hawkes process estimation:
\begin{enumerate}
\item {\bf Asymptotic distribution:} 
neither the BM-method nor the approaches mentioned in the overview present theoretical results regarding the distribution of the estimates. In contrast, we give \eqref{normal_approximation} and \eqref{cov_estimator} that open the door to confidence intervals and testing. For a particularly fertile application of our distributional results; see~\citet{kirchner16a}. Also note that in time series theory, the asymptotic distribution for YW-estimates is typically derived by noting---in a first step---that YW-estimates and CLS-estimates are asymptotically equivalent and then---in a second step---applying the asymptotic distribution of the CLS-estimates. Similarly, we can expect that the path towards the asymptotic distribution of the BM-method leads over our CLS-approach. 

\item {\bf Simplicity:} 
our method is the natural extension of the naive way to estimate general intensities of point processes as step-functions: namely normalizing the number of events in a bin by the length of the bin. Our method is also simpler in the sense that it can be explained without introducing concepts like autocovariance densities or Wiener--Hopf integral equations together with the theory of their solutions. Given the approximation insight from \citet{kirchner16c}, we only need the elementary concept of linear least-squares. Furthermore, for the BM-method, there are more choices involved that will influence the estimation: in addition to the bin-size choice $h$ and the support choice $A$ one has to choose a smoothing method for the autocovariance function estimate as well as the quadrature for the solution of the integral equation (and its parameters). In contrast, our method depends on only two parameters---with quite tractable effects at that. As a minor remark, note that the output of our calculations \eqref{notation} includes the baseline-intensity estimates whereas for the BM-method one has to do some further calculations.

\item {\bf Efficiency and bias:} 
both methods include the cut-off error, the discretization error, and the distributional error as presented in Section~\ref{intuition}. In the BM-method, by the two-step procedure, these errors seem less tractable and an additional (if small) error occurs from the Gaussian quadrature. Furthermore, in time series theory, for small sample-sizes, CLS-estimates have a smaller variance than YW-estimates because the YW-method depends on estimates of the autocovariance at large lags. These estimates typically have a very large variance as there is so little data for their estimation. The same can be expected for the point process world. Finally, it is well-known that YW- and CLS-estimates can become quite different for (multivariate) time series near unit roots. Citing from Section 4.4 in \citet{reinsel97}: `When the vector AR process is \dots\ near nonstationarity, it is known that the CLS-estimator \dots\ still performs consistently, whereas the Yule--Walker estimator may behave much more poorly with considerable bias, \dots' Given the interest in Hawkes models near criticality, e.g., in \citet{hardiman13} or \citet{rosenbaum14} this is another advantage of our CLS-based method. It would be interesting to perform a large simulation-study that compares the two methods more systematically and quantitively.
\end{enumerate}
}

\section{Refinements}\label{refinements}
Our Hawkes estimator $\hat{\bfH}^{(\Delta,s)}$ from Definition~\ref{estimator} depends on a bin size $\Delta>0$ and on a support $s>0$.  In the following section, we present procedures for sensible choices of these parameters. Furthermore, we discuss numerical and diagnostic issues.

\subsection{Choice of support}\label{choice_of_support}
\begin{figure}
\begin{center}
\subfigure[We consider a univariate Hawkes process with excitement function $h(t):= 1_{ t\leq 3} \exp(-t)$. The true support of the excitement is 3 (grey vertical lines). About 40\,000 events are simulated from this model. Following the ideas brought forward in Section~\ref{choice_of_support}, we apply automatic support-selection on this single large sample using the AIC-criterion with three different values for the preliminary bin-size $\Delta_0>0$. The value where the minimum AIC-value is attained (black vertical lines) hardly depends on $\Delta_0$ and though all three bin-sizes are rather coarse, the true support is estimated correctly up to few `$\Delta$-ticks'.
]{
{\includegraphics[width = 0.3\textwidth]{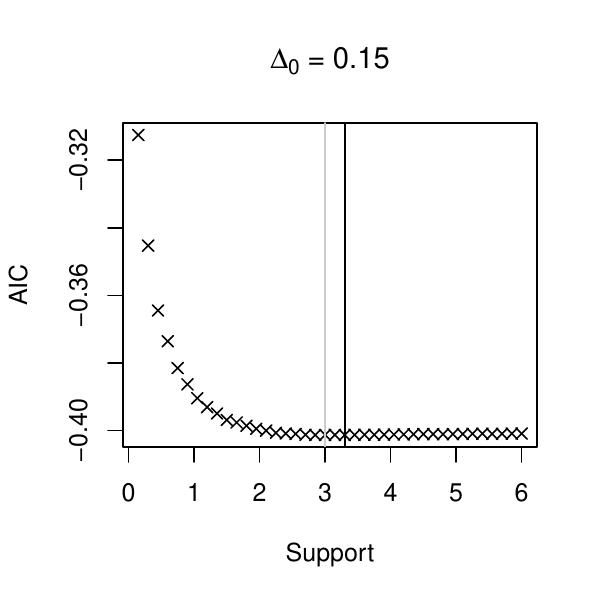}}\hfill
{\includegraphics[width = 0.3\textwidth]{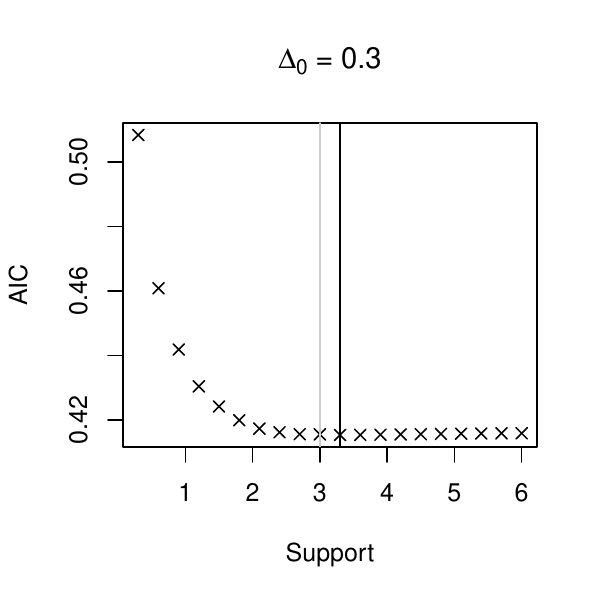}}\hfill
{\includegraphics[width = 0.3\textwidth]{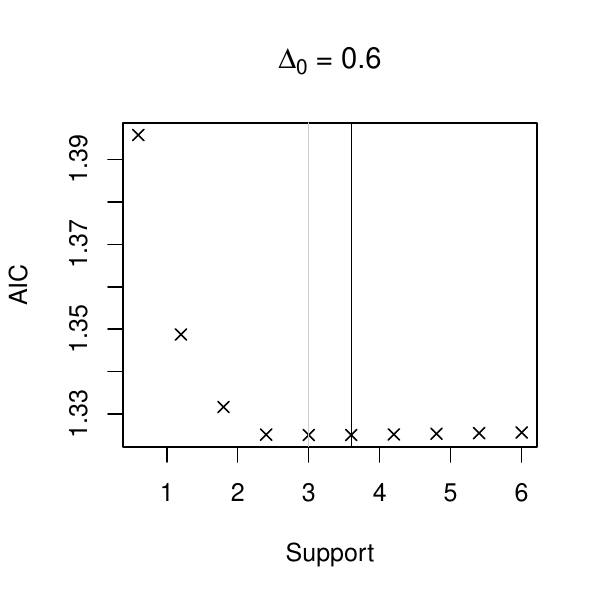}}
\label{fig4a}
}
\subfigure[We examine the infinite support case. To that aim, we consider simulated data from three univariate Hawkes process with excitement functions $h_{\alpha}(t):=0.5 \alpha\exp(-\alpha t)$, where $\alpha = 1.1 $ (circles for AIC-values and solid vertical line for AIC-minimizing support), $\alpha = 1.5$ (triangles and dashed line) and $\alpha = 2$ (crosses and dotted line). The larger $\alpha$, the lighter the tail of the function and, as desired, the smaller our estimated support; see Section~\ref{choice_of_support}. Note that the cut-off error \eqref{cutoff_error} is in all three cases so small ($<10^-3$ and much less) that it will typically be negligible in comparison to the estimation standard errors.
]{
\includegraphics[width = \textwidth]{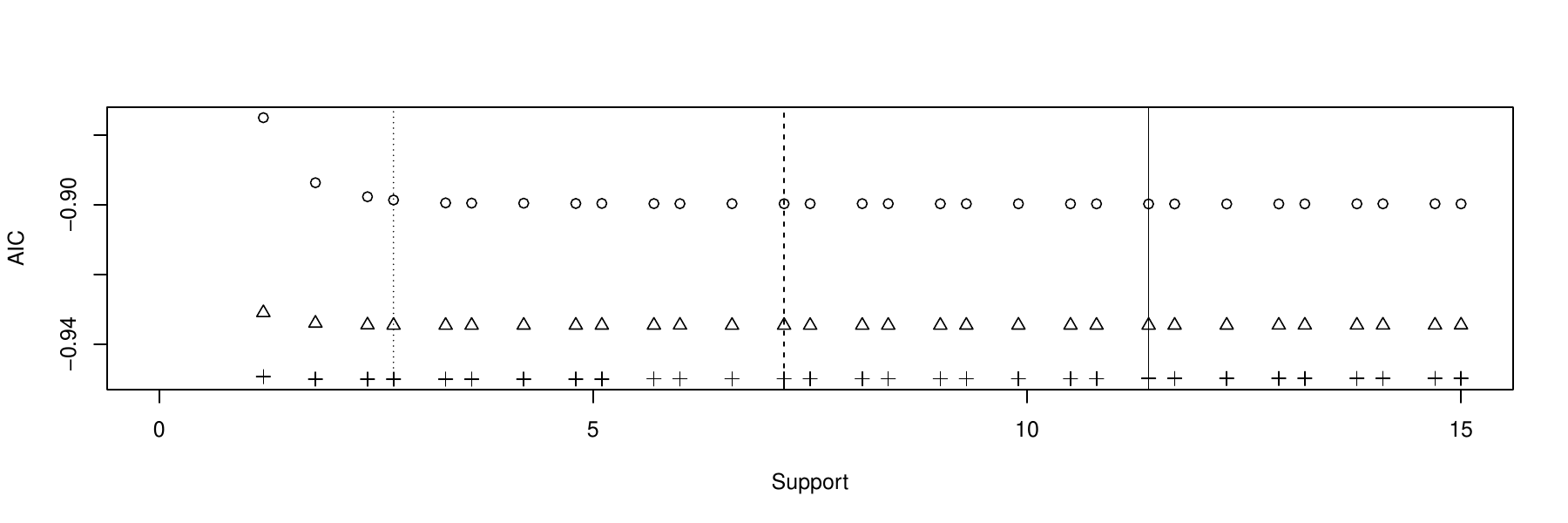}
\label{fig4b}}
\caption{
Simulation study on the choice of the support parameter of our estimator from Definition~\ref{estimator}; see Section~\ref{choice_of_support}. Figure~\ref{fig4a} illustrates the case where the true underlying support is finite and Figure~\ref{fig4b} illustrates the case where it is infinite.
}
\label{fig4}
\end{center}
\end{figure}

Estimating the support of the excitement function of a Hawkes process corresponds to estimating the largest lag of a nonzero reproduction-coefficient (matrix) of the approximating INAR sequence. In view of the VAR($p$) representation of INAR($p$) sequences from Corollary~\ref{AR-representation}, we can use any model-selection procedure stemming from traditional time series analysis; see Chapter 4.3 of \citet{luetkepohl05} for an overview of such procedures in the multivariate context. For comparison of different order-selection methods for univariate INAR($p$) sequences; see \citet{marques05}.  As a most common example, we apply Akaike's information criterion (AIC); see \citet{akaike73}. 
\par We work in the setup of Definition~\ref{hawkes_estimator}. The starting point is a sample of a $d$-variate Hawkes process on $[0,T]$ together with a preliminary bin-size $\Delta_0>0$. In our experience, a preliminary bin-size of about one event on an average per bin and component is a good choice{; see the argumentation on the distributional error in Section~\ref{choice_of_bin_size}.} With respect to this $\Delta_0$, we calculate the bin-count sequence(s) {from the data} as in \eqref{bin_count_sequence}. Now let $n_0:= \lfloor T/\Delta_0\rfloor${, $p_0 \in\N$, and let $s_0 := p_0\Delta_0>0$} be some very large support, e.g., $s_0= T/10$. Then, for{ $p \in\{ 1,2,\dots , p_0\}$}, we calculate Akaike's information criterion
\begin{align}
\mathrm{AIC}^{{(\Delta_0)}}(p) := \log \left(\det \hat{\Sigma}^{{(\Delta_0)}}(p)\right) + \frac{2 p d^2}{n_0-p},\label{AIC}
\end{align}
where $\hat{\Sigma}^{{(\Delta_0)}}(p):=\sum_{k=p+1}^{n_0} {\hat{\bfu}_k^{({\Delta_0},p)} \hat{\bfu}_k^{({\Delta_0},p)}}^\top/{(n_0-p)}$.
Here, with the notation from Remark~\ref{cov}, 
$$
\hat{\bfu}_k^{({(\Delta_0)},p)}:= \bfX_k^{(\Delta_0)} - \Delta_0\hat{\bfeta}^{(\Delta_0,{p\Delta_0})} - \sum\limits_{l=1}^p\Delta_0\hat{H}_l^{{(\Delta_0,{p\Delta_0})}} \bfX_{k-l}^{(\Delta_0)},\quad k = p+1,p+2,\dots,n_0,
$$ 
denote the estimated prediction-error vectors with estimated coefficient-matrices from the fit of the approximating INAR($p$)-model with respect to $p$ lags; see \citet{luetkepohl05} for the multivariate AIC-formula \eqref{AIC}. Finally, we choose $\hat{p}^{{(\Delta_0)}}:=\argmin_{p\leq p_0} \mathrm{AIC}(p)$ as estimated maximal lag for the approximating INAR model, respectively, we choose
$\hat{s}^{{(\Delta_0)}}:= \hat{p}^{{(\Delta_0)}} \Delta_0$ as support parameter in the calculation of the Hawkes estimator \eqref{estimator_calculation}. {The estimate $\hat{s}^{{(\Delta_0)}}$ depends on $\Delta_0$.} {However, given our approximation result in Theorem~\ref{convergence}, we expect the estimates to be quite stable over the bin-size choices. Indeed, the examples below confirm this view: for large sample-sizes, we always get that $\hat{s}^{{(\Delta_0)}} \in [s\pm 2\Delta_0)$, where $s$ denotes the true underlying (finite) support of the data generating excitement function.}
\par In parametric Hawkes setups, the support of the Hawkes excitement function is typically chosen infinite. Our estimation procedure, however, assumes finite excitement. {Note that} from \eqref{stability_criterion}, we get that the excitement functions necessarily vanish for large times. In other words, the influence of the tail of the excitement on the model is negligible; see \eqref{cutoff_error} and Proposition~\ref{INAR(p)_approximation}. 
The only question remaining is how we can choose a support $p\in\N$, respectively, $s>0$, large enough so that the truncated model with the truncated excitement is a good approximation for the true model. {We will see in the examples below that in this infinite-support case} the AIC-approach from 
above{ also turns out to be helpful}. As a side remark note that the standard parametric approach is not free from cut-off errors either, as we only observe data in finite time-windows.

%
\subsubsection*{{Example with exponential decay}} 
We simulate 
from a {univariate} Hawkes model with excitement function
$h(t):=1_{t\leq3}\exp(-t)
$
and calculate the AIC-minimizing support with respect to different preliminary {bin-sizes} $\Delta_0$. All results catch the true support very well{; the estimated results are remarkably stable with respect to the choice of $\Delta_0$.}
Next, we consider the case of infinite support {$t\mapsto0.5\alpha\exp(-\alpha t)$} with respect to three different decay parameters $\alpha\in\{ 1.1, 1.5, 2\}$. Again, we simulate large samples for each of the three $\alpha$-values and then calculate the three corresponding AIC-minimizing support estimates. The smaller $\alpha$, the larger the tail of the true excitement function becomes and---as desired---the larger the support estimates $\hat{s}^{(\mathrm{AIC})}(\alpha)$ get. For all choices of $\alpha$, the ignored excitement weights, $0.5 \alpha \int_{\hat{s}^{(\mathrm{AIC})}(\alpha)}^\infty \exp(-\alpha t)\mathrm{d}t$, are very small (less than $10^{-3}$). {That is, in the exponential-decay case, our method controls the cut-off error \eqref{cutoff_error} well}.

\subsubsection*{{Examples with power-law decay}}
{
Next, we consider the case where the excitement function is governed by an extremely slowly decaying power-law such as
\begin{equation}
h(t) := \beta (\alpha - 1) ( 1 + t )^{-\alpha}\label{power_law}
\end{equation} with $ \alpha - 1>0$, small, and $\beta  \in(0 , 1)$. We apply the AIC-based support-choice method from above to a large realization of a Hawkes process with excitement of form \eqref{power_law}. The realizations of $\hat{s}$ increase when $\alpha$ decreases (not illustrated)---as desired and just like in the (infinite) exponential-decay case. However, the cut-off error $\int_{\hat{s}}^\infty h(r)\d r$ is now typically large. E.g., for $\alpha = 1.15$ and $\beta = 0.5$, our method yields a support parameter $\hat{s}\approx 20$---quite independently of the choice of $\Delta_0$. This estimated support leaves a very large cut-off error of $\int_{20}^\infty h(t) \d t \approx 0.32$. In other words, we miss nearly two thirds of the total excitement ($=0.5$). How is it possible that our selection method opts for such a faulty model? 
The AIC-based support selection chooses $\hat{s}$ 
such that the increase of local explanatory power by looking even further back into the past is relatively small. 
In other words, despite the large cut-off error,
the truncated model $({\eta}^{(\hat{s})}, {h}^{(\hat{s})})$ is locally very similar to 
the Hawkes model $(\eta, h)$, where ${h}^{(\hat{s})}(t):=1_{t\leq s} h(s)$ and ${\eta}^{(\hat{s})} :=  \eta +\eta/(1- \beta) \int_s^\infty h(t)\d t = \eta( 1 + \beta/(1 - \beta) (1 + \hat{s})^{1 - \alpha})$. 
%
It would be interesting to discuss this truncation-approximation more quantitively as the cut-off comes with enormous computational advantages (for all estimation methods). We will touch the specific issues arising in the context of Hawkes excitement-functions with power-law decay again in Section~\ref{Bivariate_estimation}. }
\subsubsection*{{Bivariate example}}
\par Finally, we consider a bivariate Hawkes model with the excitement function $H$ from \eqref{bivariate_example}. We realize a single large sample from this model. Then we simulate another sample from a truncated version of the model, with
\begin{align*}
H^{(\mathrm{tr})}(t)=
\left(\begin{array}{cc} 
h_{1,1}(t) & h_{1,2}(t)\\
h^{(\mathrm{tr})}_{2,1}(t) & h_{2,2}(t)
\end{array}\right) 
= \left(\begin{array}{cc}0& 1_{1<t\leq 3} 0.25\\
1_{t\leq4}0.5(1+t)^{-2}\quad & 1_{t\leq \pi} 0.2 \sin(t)
\end{array}\right).
\end{align*}
The AIC-minimizing support estimate
is 9.5 for the original model and 4.2 for the truncated model. So the AIC-approach is able to discriminate between these cases.  

\subsection{Choice of bin size}\label{choice_of_bin_size}
Below, we discuss the choice of the bin size $\Delta>0$ for the Hawkes estimator $\hat{\bfH}^{(\Delta,s)}$ from Definition~\ref{estimator}. {At first sight, }one can interpret the choice of the bin size $\Delta$ as a bias/variance trade-off: the smaller $\Delta$, the smaller the potential bias stemming from the model approximation, i.e., the smaller the errors \eqref{distributional_error} and \eqref{discretization_error}. At the same time, due to the $1/\Delta$ factor in the calculation of the estimator matrix $\hat{\mathbf{H}}^{(\Delta,s)}$ from  \eqref{estimator_calculation}, its (componentwise) variance increases when $\Delta$ decreases. In a simulation study, we simulate $100$ times from a univariate Hawkes model with excitement function $h(t)= \exp(-1.1t)$. We suppose a reasonable support $s>0$ has already been chosen by a procedure as described in Section~\ref{choice_of_support}. For each sample, we calculate the Hawkes estimator with respect to three different bin sizes $\Delta\in\{0.1, 0.5, 1\}$. Figure~\ref{fig5} collects the estimation results in boxplots. The bias/variance trade-off is obvious. Note, however, that we had to choose the bin-size quite large to make the bias visible at all. {
In the following, we discuss various issues related to bin-size choice. Overall, we want to argue that one should choose the bin size $\Delta$ as small as (computationally) possible.}

 {\subsubsection*{Estimation of baseline intensities and branching coefficients}
 If we are only interested in estimates of aggregated values such as baseline-intensity components or branching coefficients, the bin-size $\Delta \downarrow 0$ leaves the variance of the estimates nearly uneffected. Indeed: let $\hat{K}^{(\Delta,s)}$ denote the branching-matrix estimate where each entry is estimated as the branching coefficient in the univariate example from Section~\ref{simulation_study}. Then we have that
\begin{equation}
\hat{\Lambda}~:=~\hat{\bfeta}^{(\Delta,s)}(1 - \hat{K}^{(\Delta,s)})^{-1}\label{average_intensity}
\end{equation}
is essentially equal to $\sum_{n = 1}^{\lceil T/\Delta \rceil} \bfX_n^{(\Delta)}/T=\bfN((0,T])/T$ and therefore \eqref{average_intensity} does not depend on $\Delta$. The same is true for $\hat{\bfeta}^{(\Delta,s)}$ and $\hat{K}^{(\Delta,s)}$: with respect to the same point process data, $\hat{\bfeta}^{(\Delta)}$ and $\hat{K}^{(\Delta)}$ hardly vary in $\Delta$. Figures~\ref{fig6} and~\ref{rev1} confirm this argumentation: if anything, the variances of baseline-intensity and branching-coefficient estimate decrease for $\Delta\downarrow 0$. Consequently, if we only want to estimate baseline intensities and branching coefficients, there is no reason why we should not choose the bin size $\Delta$ as small as computationally possible.
 }

 \subsubsection*{Estimation of the excitement function}
 {If we want to estimate specific values of the excitement functions, we face an increasing variance for decreasing $\Delta$; see the Figures~\ref{fig6} and~\ref{fig5}.} {However,} we should keep in mind that the final goal of our analysis may be the estimation of the excitement-function components $h_{ij}$---and not only {for} a finite number of their values $h_{ij}(k\Delta),\, k=1,2,\dots,p$. When we apply some smoothing method on these values, a smaller $\Delta$ typically leads to an `averaging' over more point estimates. This averaging balances the increase in pointwise variance. {At first sight, this seems an odd thing to do: the bandwidth of the smoothing method seems more or less equivalent to the bin size of the first bin-wise aggregation. In particular, it looks as if the bias that we avoided by applying a small bin-size $\Delta$ for the original estimation is reintroduced by choosing a coarser bandwidth $\tau$ for the smoothing of the estimates. This is only partly right. To understand this, we have to reconsider the errors in the approximative model equation from Section~\ref{intuition}:} {
a too large bin-size $\Delta$ effects two of the approximation errors, namely, the distributional error \eqref{distributional_error} and the discretization error \eqref{discretization_error}. The crucial observation is that the distributional error is \emph{not} reintroduced by the smoothing method. We repeat the source of the distributional error: for example, supporse that we observe three events in a bin. In the approximating bin-count model, we explain all of these three events by events in earlier bins. But the last of the three considered events is very likely a result of the two previous events in the bin in question itself! We do not account for this in our approximative model equation. Consequently, we overestimate the influence of the past on the bin (or overestimate the baseline intensity). This (important) part of the bias becomes smaller and smaller with decreasing $\Delta$ and nearly vanishes if $\Delta$ is chosen so small that no bin contains more than one event. This distributional error is \emph{not} reintroduced by any aggregating smoothing method applied on the pointwise estimates. This effect can be easily observed empirically (not illustrated): for a single large simulated sample from a Hawkes process, we calculate the pointwise Hawkes estimator from Definition~\ref{estimator} with respect to three different bin-sizes $\Delta$. Applying a cubic smoothing-spline procedure on the results we get some function estimates. The bias of these smoothed functions vanishes for $\Delta\downarrow 0$ and, at the same time, their variance does not increase.} We conclude: if the goal of the estimation procedure is a completely specified Hawkes model, then the smallest $\Delta$ that is computationally convenient may be chosen.

 \subsubsection*{Bias correction}
 {After explaining the relation between support and bin-size choice on the one side with cut-off error~\eqref{cutoff_error} and distributional error~\eqref{distributional_error} on the other side,} we next propose a heuristic bias correction idea {that mainly corrects the discretization error \eqref{discretization_error}. The method can be applied in the special (but typical) case when the underlying true excitement function is decreasing and convex.} We demonstrate the idea in the univariate case: up to now, we interpreted the $k$-th entry of  $\hat{\bfH}^{(\Delta,s)}$, i.e., $\hat{h}_k^{(\Delta,s)}$, as an estimate for $h(k\Delta)$, the excitement from an event at some time $t\in\R$ on the conditional intensity at time $t+k\Delta$; see Remark~\ref{cov}. But what $\hat{h}_k^{(\Delta,s)}$ measures is in fact an \emph{aggregation of the excitement from one bin to the $k$-th next bin}. As a consequence, rather than estimating $h(k\Delta)$, the term $\hat{h}_k^{(\Delta,s)}$ estimates more an average of the form
$$
a_k:=\frac{1}{2\Delta}\int\limits_{(k-1)\Delta}^{(k+1)\Delta} h(t)\d t,\quad k = 1, 2, \dots, p.
$$
For convex functions $h$, $a_k$ is always larger than $h(k\Delta)$. If at the same time, $h$ is decreasing, we have $a_k\approx h\left(\left(k-0.5\right)\Delta\right)$. Following these heuristics, a large part of the discretization error can be corrected by shifting the estimation grid of the excitement function by 0.5$\Delta$ to the left. In other words, if we interpret $\hat{h}_k^{(\Delta,s)}$ as an estimator for $h\big((k-0.5)\Delta\big)$, this typically corrects a large part of the bias. {This bias correction may be useful if we have to choose $\Delta$ quite coarse.}
 
\subsubsection*{{Computational issues}}  If we choose a very small bin-size $\Delta$, computation time becomes an issue. 
The calculations in \eqref{estimator_calculation} require the construction of the design matrix $\bfZ$ from Definition~\ref{CLS} with about $T/{\Delta}$ rows and about $d\cdot s/\Delta$ columns. Here, $T$ is the size of the time window, $d$ is the dimension of the process, and $s$ is the support parameter of the estimation. Then the matrix $\bfZ\bfZ^\top$ has to be inverted. This square matrix is approximately of size $\lceil d\cdot s/\Delta\rceil\times \lceil d\cdot s/\Delta\rceil$. In short, the smaller $\Delta$, the larger the matrices involved. Note, however, that, for a very small bin-size $\Delta$, 
the corresponding design-matrix is very sparse. Specialized software makes construction and manipulation of sparse matrices numerically efficient; see \citet{maechler12}. {
\par Note that if $d$ is large, we might not be able to choose $\Delta$ small enough which leaves significant bias in our estimates. In this case, we propose the following two-step procedure: first we choose a too large (but computationally feasible) bin size $\Delta_1$. With respect to this preliminary $\Delta_1$, we calculate the estimates of all branching coefficients including confidence bounds with respect to some significance level $\alpha$; see Section~\ref{simulation_study}. If the confidence intervals include zero, that is, if the corresponding excitement is not significant, we assume that there is no excitement at all. Given that the underlying true `excitement-graph' is sparse, the `excitement-graph estimate' will typically also be sparse. This reduces the dimensionality of the problem remarkably. The reduction allows to reestimate the remaining excitements in a second step \emph{with respect to a much smaller bin-size $\Delta_2$}. In this sense, the bin-size can be applied as a parameter controlling the computational complexity of the method. This two-step procedure is worked out and demonstrated in \citet{kirchner16a}. Also note that the calculation of the covariance matrix estimate \eqref{cov_estimator} can be computationally even more challenging than the calculation of the estimator itself. For that reason, we also provide particularly efficient ways for the calculation of the variance of the branching coefficient estimates in the cited paper.}

{\subsubsection*{Choosing $\Delta$ small enough} 
 We now understand that the trade-off related to the bin-size choice is not so much a bias/variance trade-off but rather a bias/computational-issues trade-off!} To check if we have chosen $\Delta$ small enough, we propose to calculate the (biased) estimate of the baseline intensity vector $\bfeta:=\left(\eta_i\right)_{1\leq i \leq d }$ for a decreasing sequence of bin sizes $\Delta_0>\Delta_1>\Delta_2>\dots$\;
 \begin{center}
\begin{figure}
\includegraphics[width = \textwidth]{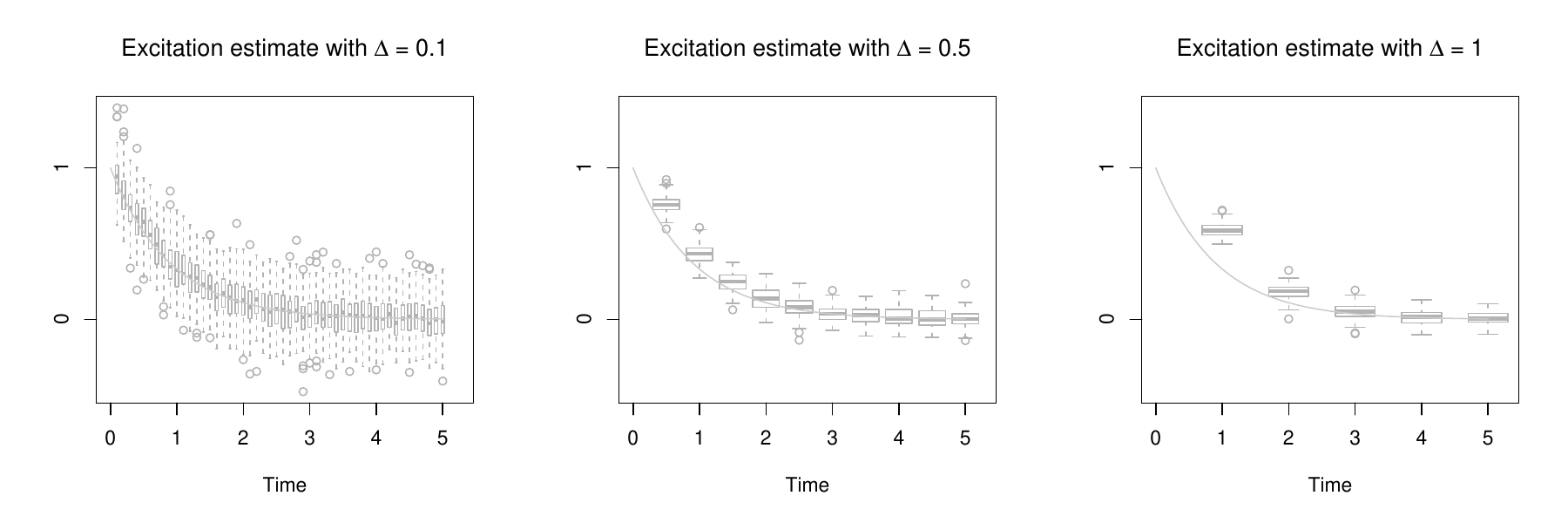}
\caption{Illustration of the bias/variance trade-off in the choice of the bin size $\Delta$; see Section~\ref{choice_of_bin_size}. We simulate 100 realizations of a Hawkes process. For each of these 100 samples, we calculate the estimator from Definition~\ref{estimator} with respect to three different bin-sizes $\Delta\in\{0.1,0.5,1\}$. The estimates are collected in boxplots. The grey lines denote the true excitement function $h(t) = \exp(-1.1t)$. A larger $\Delta$ leads to a larger bias. This is particularly obvious in the first boxplot of the right panel. Note that the bin sizes had to be chosen quite coarse to make this bias visible. A smaller $\Delta$ leads to larger pointwise variance of the excitement function value estimates.}
\label{fig5}
\end{figure}
\end{center}

The variance of the {baseline-intensity} estimates $\hat{\eta}_{i}^{({s,\Delta_n})},\,{n=0,1,\dots}$ is approximately constant over the different bin-sizes; see Figure~\ref{fig6}. This makes the estimates comparable.
For $i = 1,2,\dots, d$, we plot the values $\left(\hat{\eta}_{i}^{({s,\Delta_n})}\right)_{n=0,1,\dots}$ against $\left(\Delta_n\right)_{n=0,1,\dots}$. Typically, one observes a monotone convergence in $n$ to some constant (or $d$ constants for $d>1$). 
Plotting confidence intervals around the point estimates indicates when the bias is negligible in comparison to the random noise of the estimate. We will apply this method in the concluding data-example.

 \subsection{Diagnostics} \label{diagnostics}
We see a certain danger in the application of our nonparametric Hawkes estimator from Definition~\ref{estimator}. Reasonable graphical results as in Figure~\ref{fig2} might be used as an argument in favor of the Hawkes process as the true model. But this conclusion would be a misuse of the method. In fact, the proposed estimator depends only on second-order properties of the data. So, we have to expect that there is a whole family of point processes that generate the same excitement estimates, although only one of these processes is a genuine Hawkes process. As an example, consider a continuous-time, nonnegative, stationary Markov chain that has the same second-order properties as some given Hawkes process. We use this Markov chain as a stochastic intensity for another point process; see \citet{daley03}, Example 10.3(e). The resulting doubly-stochastic point process is a point process with different distributional properties than the corresponding Hawkes process. But our estimator will still yield the same results in both cases. As another example, consider a time-reversed Hawkes process. Clearly, this is not a Hawkes process anymore. However, the time-reversed version has the same autocovariance density as the original process and therefore our estimator will again yield the same result.
\par This means, the application of our estimation approach always ought to be followed by a model test. A most common basis for such a test in our context is a multivariate version of the random time-change theorem for point processes; see \citet{meyer71, brown88}: for points $\big(T^{(i)}_{k}\big)_{k\in\Z} ,\, i = 1,\dots,d,$ from a $d$-variate point process with conditional intensity $\Lambda = \left(\Lambda^{(i)}\right)_{i=1,\dots ,d}$, one has that
$\int_{T^{(i)}_{k}}^{T^{(i)}_{k+1}} \Lambda^{(i)}\left(t\right) \mathrm{d}t {\sim}\operatorname{Exp}(1)$ independently over $i= 1,\dots,d$ and $k \in\Z .$
So, after having fit the Hawkes process to point process data, we calculate the corresponding conditional-intensity estimate and time-transform the interarrival times. These transformed interarrival times ought to be compared with theoretical Exp(1)-quantiles in a QQ-plot. Next to this graphical method one ought to apply a Kolmogorov--Smirnov test and an independence test to the transformed interarrival times. 


\section{Data application}
 There are two contexts of growing importance where large event-data sets 
are not the exception but 
the rule: internet traffic and high-frequency data in financial econometrics. 
The paper concludes with an exemplary application of the estimation procedure to  the latter.

\subsection{The data}\label{data}
\begin{figure} \center
\includegraphics[width = 0.8\textwidth]{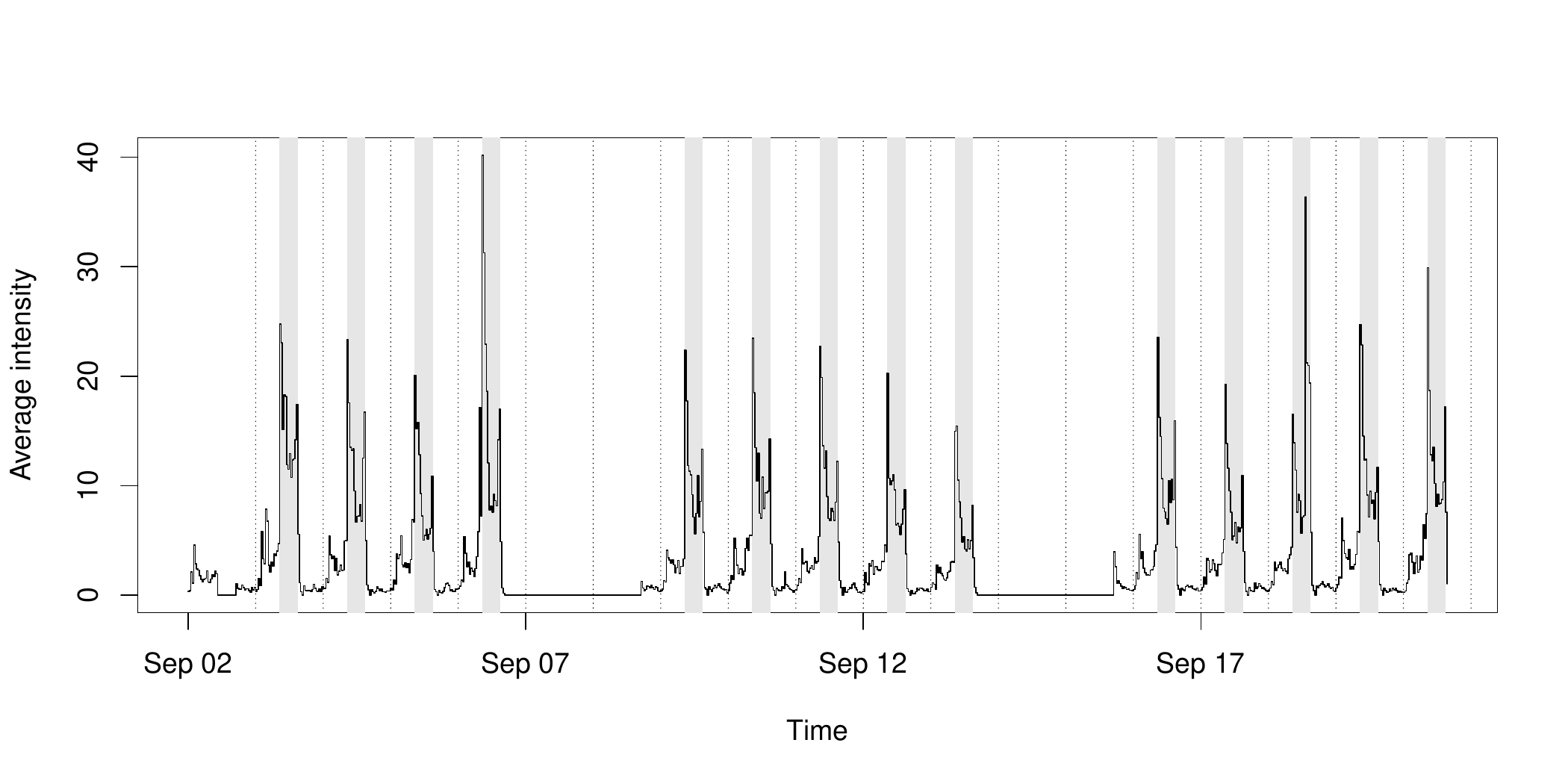}
\caption{
{
Illustration of our data set; see Section~\ref{Bivariate_estimation}. The time is Chicago local time. The dotted vertical lines refer to midnights. The black solid line shows the average number of order-book events per second in $30\min$ windows, as explained in Section~\ref{data}. For our estimation, we only consider the regular trading hours, that is, 8:30am--3:00pm, Chicago time (grey stripes). Inside these stripes, we observe the characteristic U-shape of the intensity. Also note the two preceding smaller U-shapes. These correspond to the regular trading hours of the exchanges in Europe and Asia. The vanishing activity on the first Monday of our data is due to a US-holiday (Labor Day) on September 2, 2013.
 }
} \label{unconditional_intensity}
\end{figure}
The data we use stem from the \emph{limit order book (LOB)} of an electronic market. LOBs match buyers and sellers of a specific asset. We will consider a certain future contract. Whoever wants to buy or sell one or several of these contracts has to send his or her orders to the LOB. An order basically consists of two pieces of information: it names (a) the maximal (respectively, minimal) price at which the sender is willing to buy (respectively, to sell) and (b) the desired quantity in terms of numbers of contracts. If the order is matched to another order, the trade is executed. Such orders that immediately find counterparts are called \emph{market orders}. All other incoming orders are stacked in the LOB; these are called \emph{limit orders}. Limit orders either wait for getting executed by a new incoming matching (market) order or---and this happens relatively often---they are withdrawn after some time. The empirical process of time points when orders arrive we call \emph{order flow}. Such an order flow can be modeled by a point process. In particular, our estimation method from Definition~\ref{estimator} allows to analyze the order flow
in a Hawkes setup. For a detailed survey of order-book quantitative analysis; see \citet{gould13}. 
Financial intraday histories are attractive for econometric research as there is so much data available. 
However, by the very differing data qualities, results are sometimes hard to compare. 
To clarify our starting point, we explain the context and the preparation of the data quite detailed. 
\par We consider a sample of the LOB of E-mini S\&P 500 futures with most current maturity.
The enormous liquidity makes the data attractive for quantitative analysis. 
Samples of these particular data have also been analyzed in the Hawkes setting, e.g., by 
\citet{filimonov12} and 
\citet{hardiman13}.
Our particular data sample was provided by TickData inc. It stems from September 2013. We have a separate data set for quotes and for trades. 
A new entry in the quotes data corresponds to one of the following three events:
\begin{itemize}
\item[(i)] Arrival of some (not marketable) limit order
\item[(ii)]Arrival of some market order, i.e., a trade takes place
\item[(iii)] Cancellation of some limit orders
\end{itemize}
In the trade data set, we see the traded price and the number of contracts traded. 
In both data sets, we observe ties, i.e., multiple events with identical millisecond time-stamps. These ties require special consideration as our model, the Hawkes model, does not allow for simultaneous jumps.
As data is so relatively sparse, the multiple events cannot be accidental. This leaves two possibilities: either the multiples stem from a single order that has been split (for some technical reason) or the multiples are almost instantaneous responses to each other that are reported at the same millisecond due to rounding. {We} had the opportunity to compare our data with a snapshot of the fully reconstructed LOB. This complete data provide `match tags'  for each order. This additional information shows that nearly all multiple events are in fact orders from one single market-participant. This confirms our point of view. 
We therefore consider each time stamp in the data sets only once.
After the reproduction procedure, we derive two one-dimensional event data sets from our data: 
\begin{itemize}
\item the \emph{trade data} $\mathcal{T} 
$ and
\item the (pure) \emph{limit-order data} $\mathcal{L}$ that collects all the times when a new non-marketable limit order has arrived or a limit order has been canceled. 
\end{itemize}

\FloatBarrier

In regular trading hours, i.e., between 8:30am and 3:15pm (Chicago local time), we observe
about 5 events per second in the trade data $\mathcal{T}$, and about 12 events per second in the limit-order data $\mathcal{L}$. At Chicago night time, all of these average intensities are up to twenty times smaller. All interarrival-times processes exhibit significant autocorrelation at large lags. This rules out simple standard homogenous Poisson point processes as models as well as other renewal processes. On the other hand, the autocorrelation may also stem from nonstationarities in the underlying true model; see \citet{mikosch00}. 
 \subsection{Bivariate estimation of the market/limit order process}\label{Bivariate_estimation}

With our nonparametric method from Definition~\ref{estimator}, we fit a bivariate Hawkes process $(N^{(\mathcal{T})}, N^{(\mathcal{L})})$ {on $30\min$-samples of the data $(\mathcal{T} ,\mathcal{L})$ in the first three weeks of September 2013. We first illustrate our approach in detail for} a single $30\min$-sample, namely on data from Friday, 2013/09/06, 10:00am--10:30am (Chicago time). In this specific sample, we observe about 20\,000 trades and 40\,000 limit orders. Our estimation procedure from Section~\ref{hawkes_estimator} depends on a choice of support and on a choice of bin size. For a sensible choice of these parameters, we apply the methods from Sections~\ref{choice_of_support} and~\ref{choice_of_bin_size}:

\subsubsection*{{Choice of support}}
 \begin{figure} \center
\subfigure[Support analysis with respect to a very coarse preliminary bin-size $\Delta_0 = 1\,\mathrm{sec}$; see Section~\ref{choice_of_support}. The estimator from Definition~\ref{estimator} is calculated for different support candidates (in seconds). The corresponding AIC-values are calculated as in \eqref{AIC}. We establish a quite short AIC-optimal support of the excitement function. ]{\includegraphics[width = 0.3\textwidth]{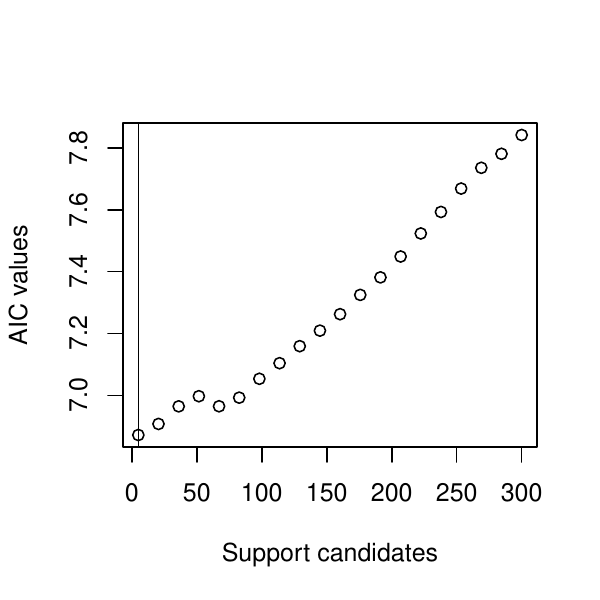} \label{td1a}} \hfill
\subfigure[After the rough analysis illustrated in Figure~\ref{td1a}, we repeat the procedure for smaller support candidates and with respect to a much finer bin-size $\Delta = 0.01\,\mathrm{sec}$. We find an AIC-optimal support value of about 2.7 seconds. This value is stable over other choices of the bin-size. The attained minumum is remarkably clear-cut compared to the attained minimum in the simuation study illustrated in Figure~\ref{fig4}.]{\includegraphics[width = 0.3\textwidth]{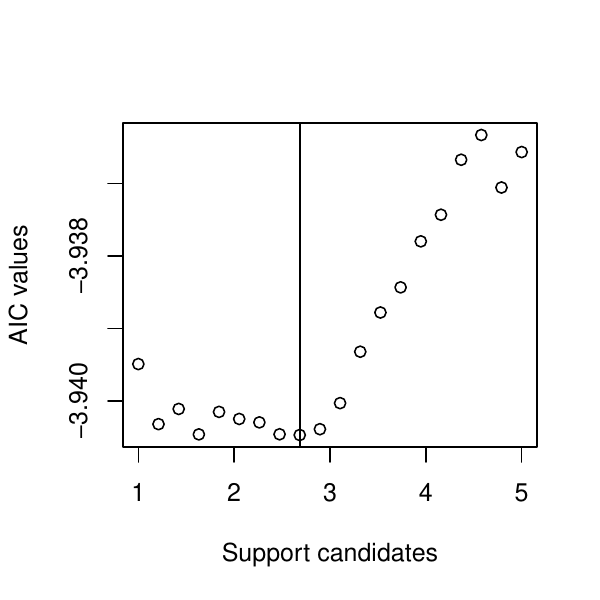}\label{td1b}}
\hfill
\subfigure[Bin-size analysis following the method from Section~\ref{choice_of_bin_size}. The baseline estimates decrease in both components as the applied bin-sizes decrease. For $\Delta$ smaller than $0.01\,\mathrm{sec}$, the decrease is of a lower magnitude than the 95\%-confidence intervals. We conclude that, for $\Delta \leq 0.01\,\mathrm{sec}$, the bias of our estimation method becomes negligible.]{\includegraphics[width = 0.3\textwidth]{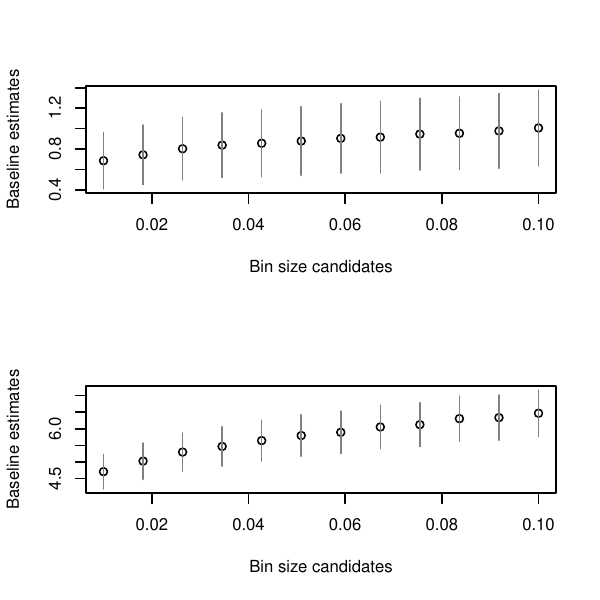}\label{td1c}}
\caption{Preliminary analysis for the bivariate data example $(\mathcal{T},\mathcal{L})$ (trades/limit orders); see Section~\ref{Bivariate_estimation}. Our nonparametric Hawkes estimator from Definition~\ref{estimator} depends on a support parameter $s$ and a bin-size parameter $\Delta$. Applying the selection methods 
from Section~\ref{choice_of_support} and Section~\ref{choice_of_bin_size}, we find that
$s = 3\,\mathrm{sec}$ and $\Delta = 0.01\,\mathrm{sec}$ are reasonable choices.
}
\label{support_analysis}
\end{figure}
\begin{figure} \center
\subfigure[Bivariate fit with respect to bin size $\Delta = 0.01\,\mathrm{sec}$ and support $s =3 \,\mathrm{sec} $. For the derivation of these estimation parameters; see Figure~\ref{support_analysis}. Eyeball examination reveals local maxima in the lower panels at half seconds. $\log_{10}/\log_{10}$-plots of averaged estimates are given in Figure~\ref{loglog_plots}.
]{\includegraphics[width = 0.98\textwidth]{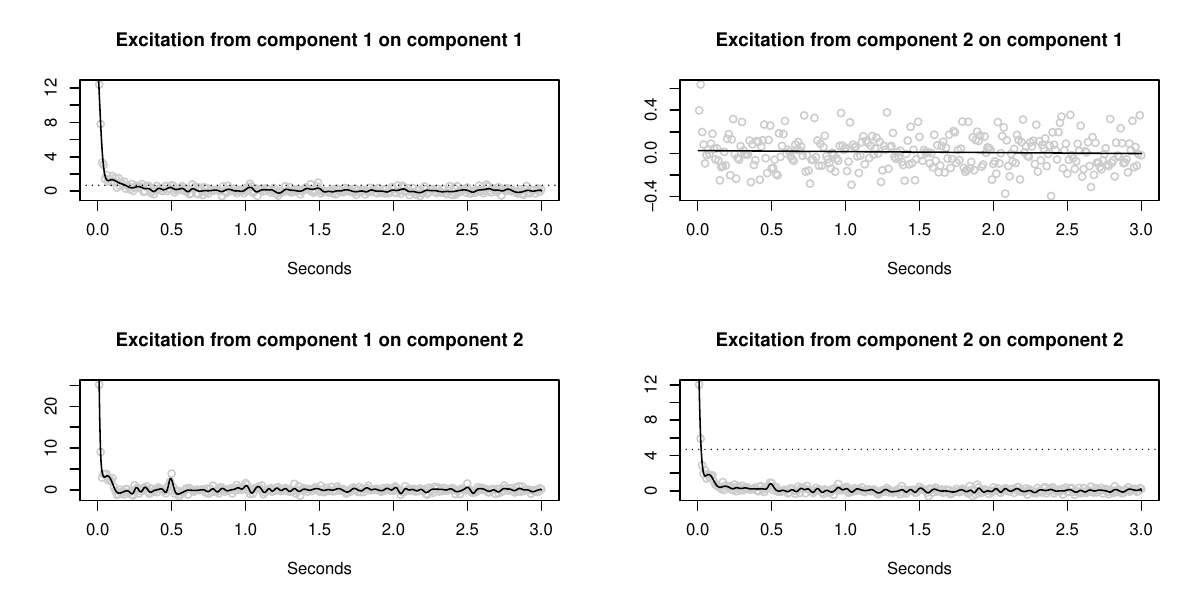} \label{td2a}}
\subfigure[We fit the Hawkes model to the same sample as in (a). This time however, we ignore the best support choice and set it naively to $s = 0.1\,\mathrm{sec}$ only. In addition, we apply an extremely small bin size of $\Delta = 0.002\,\mathrm{sec}$. 
In the first milliseconds after each event, the results indicate an inhibitory effect; the Hawkes model does not allow for negative excitement. 
In the smoothed function-estimates, we detect a local maxima at $0.01\,\mathrm{sec}$. For  the estimated excitement function of the first component (the limit-order process) we observe further local maxima at multiples of $0.02\,\mathrm{sec}$. 
 ]{\includegraphics[width = 0.98\textwidth]{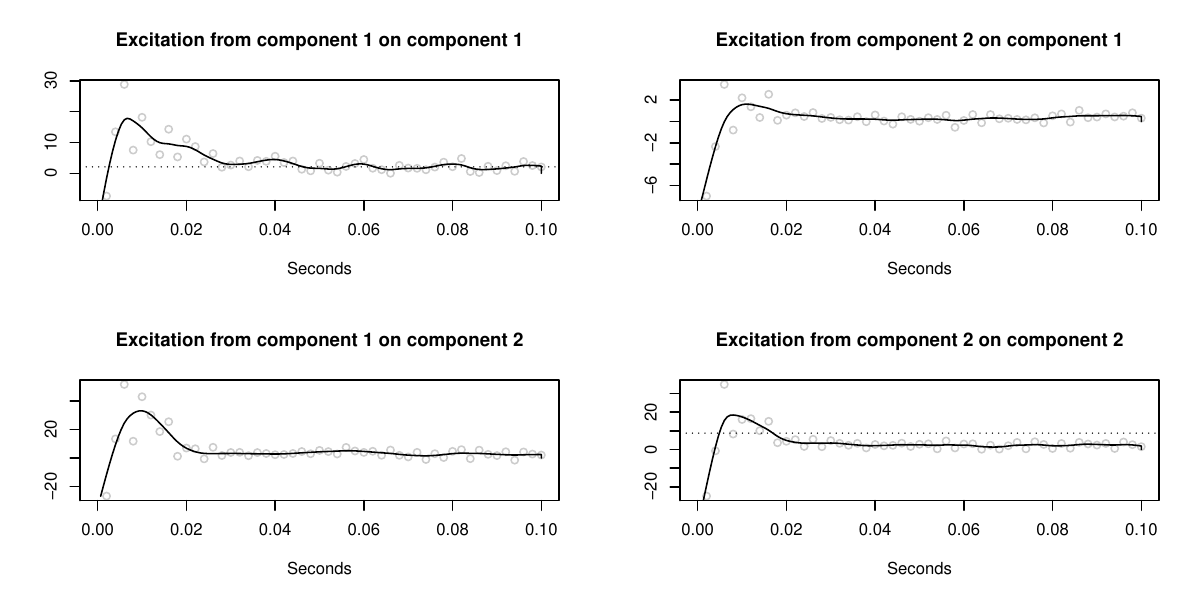}\label{td2b}}
\caption{Exemplary biviariate Hawkes fits on a single $30\min$-window; see Section~\ref{Bivariate_estimation}. We apply two sets of estimation parameters ($s$, $\Delta$).  In the fitted bivariate process, the first component refers to the trade times $\mathcal{T}$ and the second component to the limit order arrivals, respectively, cancellations $\mathcal{L}$. The black solid lines are kernel-smoothed versions of the estimates; see the end of Section~\ref{hawkes_estimator}. The dotted lines in the diagonal plots refer to the fitted baseline-intensity components. 
} \label{td2}
\end{figure}

\par As a first step, we calculate the Hawkes estimator with respect to a relatively large preliminary bin-size of $\Delta_0 = 0.5\,\mathrm{sec}$ and for various support candidates between 1 and 300 seconds. As proposed in Section~\ref{choice_of_support}, we compare the corresponding AIC-values. This coarse analysis shows that the AIC-optimal support is surely less than 20 seconds{; see Figure~\ref{td1a}}. Repeating the analysis with respect to a much finer bin-size $\tilde{\Delta}_0=0.01\,\mathrm{sec}$ on the interval $(0\,\mathrm{sec},20\,\mathrm{sec})$, we find an AIC-minimizing support of about $2.8\,\mathrm{sec}${; see Figure~\ref{td1b}}.
Let us note that the obtained minimum is much more clear-cut than in the controlled simulation study from Section~\ref{choice_of_support} illustrated in Figure~\ref{fig4}.  We set $s = 3\,\mathrm{sec}$. In other words, our support analysis indicates that the process forgets its past after three seconds. This preliminary result is already interesting: it can be interpreted such that---in this sample---the algorithms that drive the market take not more than the last three seconds of the LOB-history into account.
\subsubsection*{{Choice of bin size}} For a reasonable choice of the bin-size parameter $\Delta$, we apply the method from Section~\ref{choice_of_bin_size}. That is, we examine the impact of the bin-size choice on the estimation. We leave the support $s=3\,\mathrm{sec}$ fixed and, for different bin-size candidates $\Delta$, we calculate the baseline-intensity estimate $\hat{\bfeta}^{(\Delta)}_i,\, i=1,2,$ together with the corresponding confidence intervals; see \eqref{estimator_calculation} and \eqref{normal_approximation} for the necessary calculations. We observe a monotone relation between the bin-size candidates and the corresponding baseline-estimates. However,  for $\Delta \leq 0.01 \,\mathrm{sec}$, the differences of the estimates are of a lower order than their (estimated) confidence intervals{; see Figure~\ref{td1c}}. So it is sensible to assume that, for this particular sample, the bias of our estimation method becomes negligible for bin-size choices of $\Delta \leq 0.01\,\mathrm{sec}$.

\subsubsection*{{Estimation results for single time window}} From the bivariate event data set, we finally calculate the Hawkes estimator from Definition~\ref{estimator} with respect to support $s = 3\,\mathrm{sec}$ and bin size $\Delta = 0.01\,\mathrm{sec}$. Figure~\ref{td2a} summarizes the estimation results for this specific time thirty minute window.

 The baseline intensity of the limit-order process $\mathcal{L}$ is about four times larger than the baseline intensity of trades process $\mathcal{T}$. In both processes, we observe a strong and quite similar selfexcitement. The crossexcitement, however, is obviously directed:  we observe a very strong crossexcitement from $\mathcal{T}$ on $\mathcal{L}$, but hardly any effect from $\mathcal{L}$ on $\mathcal{T}$. The estimated interactions can be summarized in the branching-matrix estimate
\begin{align}
  \left(\begin{array}{cc}
  0.62 (\pm0.04) & 0.03 (\pm 0.01)\\
  0.55 (\pm0.06)& 0.54 (\pm 0.03)
  \end{array}\right), \quad \text{i.e.,}\quad
  \mathrel{\raisebox{0.43cm}{\text{`}}}
   \left(\begin{array}{cc}
  \mathcal{T} \stackrel{0.62}{\rightsquigarrow} \mathcal{T} &\mathcal{L}\stackrel{0.03}{\rightsquigarrow}  \mathcal{T} \\
 \mathcal{T} \stackrel{0.55}{\rightsquigarrow}  \mathcal{L} & \mathcal{L}\stackrel{0.54}{\rightsquigarrow}  \mathcal{L} 
  \end{array}\right)
   \mathrel{\raisebox{0.43cm}{\text{'}}}.
  \label{bm}
\end{align}

 See Remark \ref{cov} for the calculation of the point estimates as well as the 95\%-confidence bounds of the branching-matrix components. Also see the explanations after \eqref{stability_criterion} for the interpretation of the branching-matrix that is indicated in the right matrix. If the underlying true excitement functions are heavy-tailed, then the estimates \eqref{bm} may extremely depend on the choice of the support parameter $s$; see the simulation study described after~\eqref{power_law}. This means, that these kinds of estimates have to be interpreted (and communicated) together with the chosen value for $s$. In any case, the values and confidence bounds are a good description for the local excitement, i.e., for the influence on the intensities from the past $s$ time units. The largest eigenvalue of matrix \eqref{bm}, i.e., the  stability-criterion estimate, is 0.72. 
The strong asymmetry in~\eqref{bm} may be interpreted such that the trades cause the limit orders (and cancellations) and not vice versa. In further analysis, we found that the estimated branching-matrix, and in particular the asymmetric crossexcitement, is quite stable over all thirty minute windows of the regular trading hours (not illustrated). {Clearly, the knowledge of the distribution of the entries in \eqref{bm} is attractive beyond confidence intervals: it allows testing for causal connections between event streams. This is particularly important in higher dimensional point-process networks; we demonstrate this possible application of our estimation method in \citet{kirchner16a}---with special hindsight on computational issues and implementation. We will also apply this testing-idea to much larger data sets of limit order book data in \citet{kirchner16b}}. In the crossexcitement from $\mathcal{T}$ on $\mathcal{L}$, we observe local maxima at half and whole seconds. This effect may have two causes:  it reflects a preference either for absolute or for absolute round times. To put it differently: some of the order-sending algorithms that indeed react on trade events may have an implemented lag of half or full seconds.
 \par
 In a second approach, we fit the Hawkes model to the same sample as above{; see Figure~\ref{td2b}}. This time however, we ignore the best support choice and set it naively to $s = 0.1\,\mathrm{sec}$ only. In addition, we apply an extremely small bin-size of $\Delta = 0.002\,\mathrm{sec}$. 
In the first milliseconds after each event, the results indicate an inhibitory effect; the {standard linear }Hawkes model does not allow for negative excitement because it could yield negative intensities with positive probability. In the smoothed function-estimate of the selfexcitement of the first component (the trades process), we detect local maxima at $0.02\,\mathrm{sec}$ multiples. Also note that in this naive fit, the baseline-intensity estimates are much larger than in the first fit: these large values are a compensation for the too small support choice.
\par
Naturally, the fitted Hawkes model is only completely specified when we smooth the results from the estimation method on the grid by some kind of smoothing mechanism that yields a function $\hat{H}:\ \R_{\geq 0}\to\R^{2\times 2}$.
We do this with a cubic smoothing spline method. Having thus completely specified the model, we apply a  Kolmogoroff--Smirnov test on the transformed interarrival-times; see Section~\ref{diagnostics}. The test rejects the fitted model for the $30\,\mathrm{min}$-window. This is not surprising: given the very large sample-size, we are very likely to include `abnormal' interarrival times that our model cannot catch; the Kolmogoroff--Smirnov test is particularly sensitive to such outliers. Dividing the $30\,\mathrm{min}$-windows into smaller samples of 100 events yields plausible \emph{p}-values (not illustrated). For further interpretation of the diagnostics; see the discussion in Section~\ref{interpretation} below.
%
\FloatBarrier

\subsubsection*{{Estimation of whole data set}}

\begin{figure} \center
{\includegraphics[width = 0.8\textwidth]{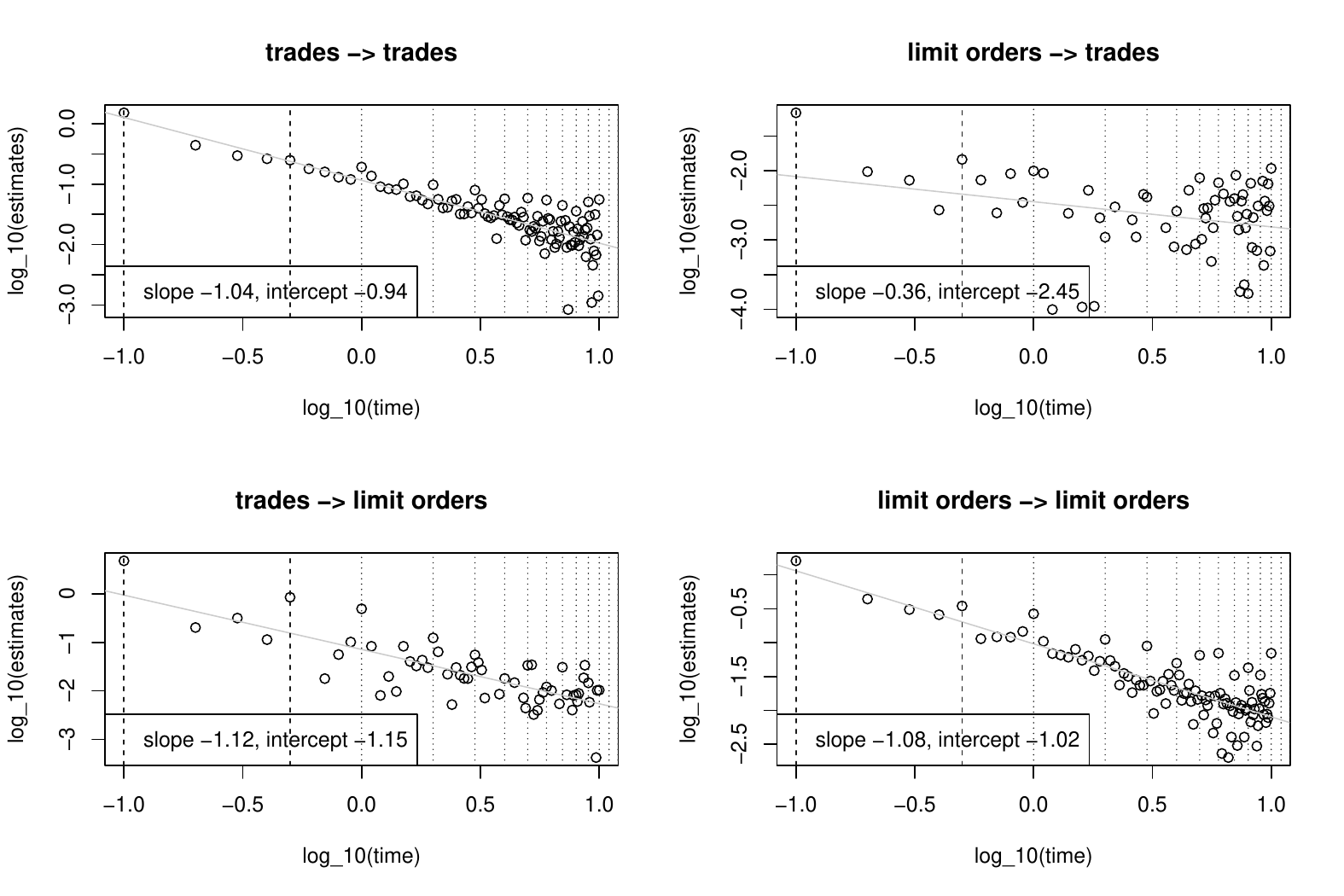} }
\caption{{Log/log plots of the exitement estimates averaged over the regular trading hours of the three week data sample; see Section~\ref{Bivariate_estimation}. The dotted vertical lines refer to seconds.
Apart from the upper-right panel ($\mathcal{L} \rightsquigarrow \mathcal{T}$),
the plots are very well linearly approximated. In the same three excitements we observe second periodicities (dotted lines) as well as peaks at $0.1\sec$ and $0.5\sec$ (dashed lines). Note that due the noise in the original estimates, we have some estimates $< 0 $ (especially for large times); these estimates are ignored in the log-transformation. This introduces a bias for the slope estimate: the true decay is presumably faster than the slope estimates indicate.}}\label{loglog_plots}
\end{figure}
{Next, we consider the whole data set; see Figure~\ref{unconditional_intensity}. More specifically, we consider thirteen $30\min$-windows in the regular trading hours of the fourteen considered trading days. The support analysis of the windows yields AIC-optimal supports between 2 and 50 seconds, the majority (and the median) being approximately 10 seconds. We first calculate the Hawkes estimator from Definition~\ref{estimator} with respect to $s= 10\sec$ and $\Delta =  0.01\sec$. We average the estimates over all $13\cdot 14 = 182$ time windows. The average branching matrix is
$$
  \left(\begin{array}{cc}
  \mathcal{T} \stackrel{0.64}{\rightsquigarrow} \mathcal{T} &\mathcal{L}\stackrel{0.02}{\rightsquigarrow}  \mathcal{T} \\
 \mathcal{T} \stackrel{0.75}{\rightsquigarrow}  \mathcal{L} & \mathcal{L}\stackrel{0.59}{\rightsquigarrow}  \mathcal{L} 
  \end{array}\right)
 $$
which is quite similar as in the single-window analysis \eqref{bm}. In particular, we observe a similar asymmetry in the crossexcitements. The log/log-plots of the estimates indicate that the excitement functions decay with a power-law with decay parameters close to 1; see Figure~\ref{loglog_plots}. The log/log-plots of the results exhibit strong second-periodicities. Also note that there are local maxima in the excitements after $0.1\sec$ and $0.5\sec$. We  checked the absolute event-times in our data set for (absolute) round-time preferences. Statistical tests indicate such preferences. But we think that they are too weak to explain the strong periodicities in the excitement.

\begin{figure}
\center
{\includegraphics[width = 0.8\textwidth]{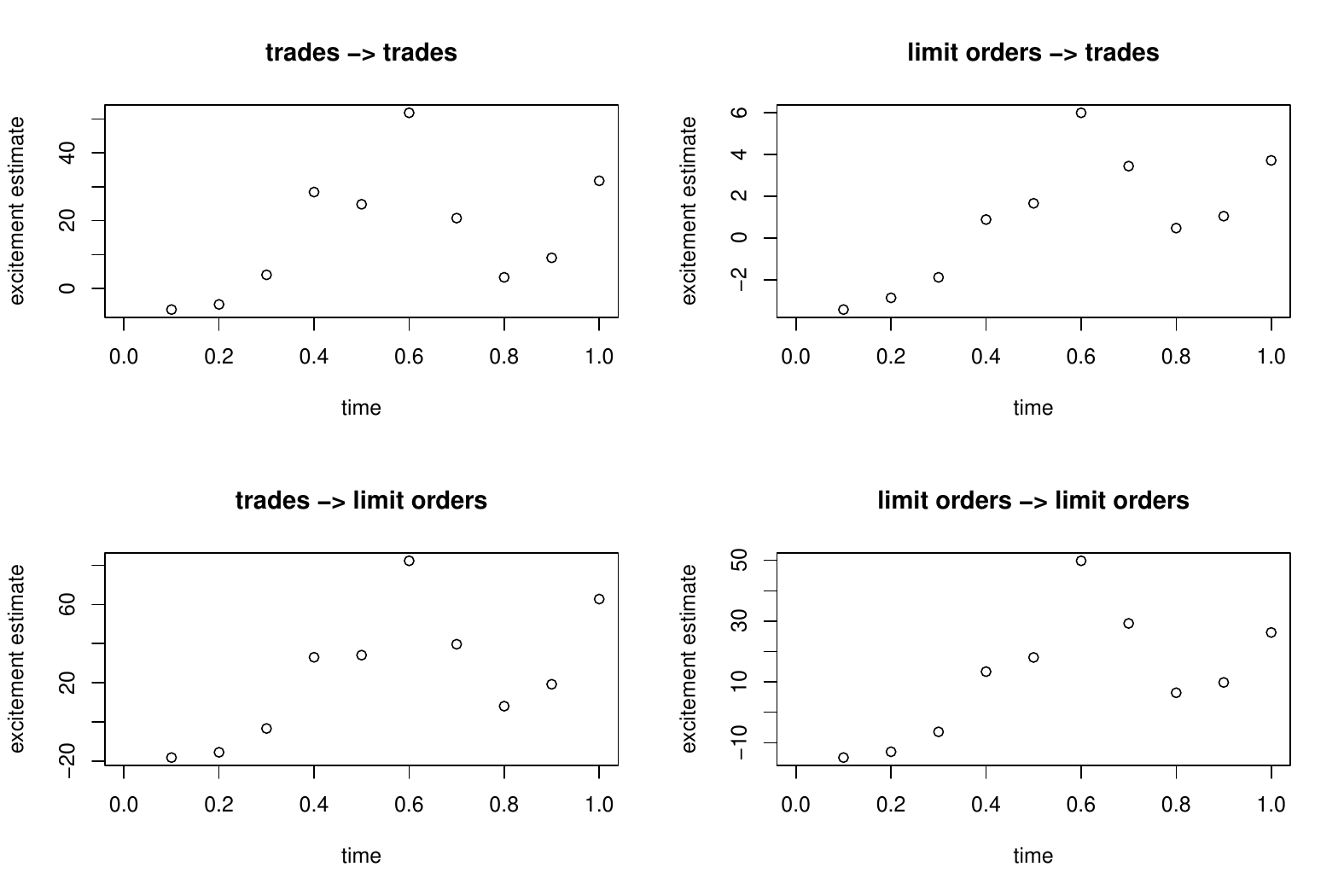} }
\caption{{Average instantaneous excitement over the 182 considered 30min-windows. We fit the Hawkes model naively with a support parameter $s = 0.01\sec$ and a super-small bin-size of $0.001\sec$. Note that the resolution of our original data is milliseconds, so there is no data aggregation involved anymore. Consequently, it might be more appropriate to speak of an INAR(10) process rather than a Hawkes process. The shapes of the four estimated excitements are remarkably similar: there are negative values up to $0.003\sec$, a maximum at $0.006 \sec$, and another local maximum at $0.01\sec$. The excitement delay is presumably related to the \emph{high-frequency cut-off} as discussed in \citet{hardiman13}.}
} \label{ultrashort_excitement}
\end{figure}

\par Then---just like for the single-window analysis above---we study the instantaneous excitement in more detail by setting $s$ as small as $0.01$ and $\Delta$ as small as $0.001\sec$---which is the resolution of the original data. This very short excitement looks quite noisy (e.g. not monotone) on the one hand, on the other hand, it seems remarkably similar for all four excitements: we observe negative values up to $0.002\sec - 0.003\sec$,  a maximum in $0.006\sec$, and another local maximum at $0.01\sec$. This delay in the excitements may be a manifestation of the \emph{high-frequency cut-off} discussed in \citet{hardiman13} or the \emph{average latency} discussed in \citet{bacry14a}. We think that the observed inhibitory effects are statistical artifacts that are a compensation for some misspecification of the model. It would be interesting to examine this ultra-short behavior in more detail. Also note that when we get that close to the resolution of the data it might be more appropriate to apply the discrete-time INAR model directly instead of using the continuous time Hawkes model.
\par We also examine the development of the model over the day. To that aim, we illustrate  the branching-coefficients estimates and the average baseline-intensity estimates over the fourteen 8:30am--9:00am time windows, over the fourteen  9:00am--9:30am time windows, and so on. The estimates are calculated as proposed before, that is, with respect to $s = 10\sec$ and $\Delta = 0.01\sec$. The branching-coefficient estimates are quite stable over the regular trading hours---with outliers in the time windows right after the opening and right before the daily trading-stop (at 3:15pm). The average baseline-intensity exhibits the typical U-shape of the empirical intensity from Figure~\ref{unconditional_intensity}.
\par
Note that---due to the power-law shape of the excitement---choosing a larger support-parameter will typically yield larger branching-coefficients. In this case, the branching-coefficient estimates from above are somewhat arbitrary (compared to the case of finite or exponentially decaying excitement functions). Still they are meaningful \emph{in combination with the applied $s = 10\sec$}: an estimated branching-coefficient may be interpreted as the number of (direct) offspring of an event \emph{in the next $10$ seconds}; respectively, it measures the influence of an event on the corresponding intensity component in the next $10$ seconds. 
\FloatBarrier
\subsubsection*{Semiparametric fit of power-law parameters}
 An alternative way to infer the branching coefficients would be to conclude from Figure \ref{ultrashort_excitement} that the model is governed by a parametric power-law function of the form
$
h(t) = \beta 1_{t \geq \varepsilon} t^{-\alpha}
$
for some very small $\varepsilon>0$ and $\alpha >1$. If slope and intercept are estimated coefficients from the linear model explaining the $\log_{10}$-transformed estimates with the corresponding log-transformed times, then we have that $\alpha\approx-\text{slope}$ and $\beta\approx 10^{\text{intercept}}$. Consequently, we get
\begin{equation}
\int_0^\infty h(t) \d t = \frac{\beta}{\alpha - 1} {\varepsilon^{1 - \alpha}}\approx 
-\frac{10^{\text{intercept}}}{1 + \text{slope}}\varepsilon^{1 + \text{slope}}.\label{bm_est}
\end{equation}
In principle, we can read-off the high-frequency cut-off $\varepsilon$ from Figure~\ref{ultrashort_excitement}.  Say we set $\hat{\varepsilon}= 0.003$. If we plug in estimates for slope, intercept, and high-frequency cut-off, we obtain a semiparametric estimate for the branching coefficients. For the self-excitement of the limit orders in the lower right panel of Figure~\ref{loglog_plots}, for example, \eqref{bm_est} yields a branching coefficient of approximately 1.9. This value indicates a supercritical selfexcitement which would mean that the data-generating process is nonstationary and (in the long run) explosive. This result may indicate that the infinite extrapolation of the excitement is wrong. And even if the true underlying model were truly governed by a power-law excitement, there would be problems with the estimation approach derived from \eqref{bm_est}: 
\begin{enumerate}
\item the (nonlinear) log transformation introduces (additional) bias; 
\item the branching-coefficient approximation \eqref{bm_est} is extremely sensitive on the value of $\varepsilon$, which itself is difficult to estimate;
\item some of the estimated excitement-values $\hat{h}_k, {k = 1,\dots, p}$ might be negative due to noise---especially for large $k$. But then, these negative estimates get lost in the log-transform whereas their positive counterparts survive. This introduces an upward bias for large $k$ and, consequently, the decay that we observe in the log/log representation is typically slower than the true decay.
\end{enumerate}
In view of the above, the estimation approach based on~\eqref{bm_est} only makes sense for very large sample-sizes. In most cases, it is more sensible to estimate power-law parameters directly from original (untransformed) estimates via nonlinear least squares. Consider once more the selfexcitement of the limit orders in the lower-right panel of Figure~\ref{loglog_plots}. Applying nonlinear least-squares optimization to the original estimates (e.g., with {\tt nls()} in {\tt R}) yields a decay-parameter estimate of $1.25$. This indicates a rather faster decay than the one resulting from the linear log/log-fit ($1.08$).
}

\subsubsection*{{Power-law extrapolation vs. truncation}}
{
Next to the estimation issues discussed above, there are conceptual problems with extremely slowly decaying excitement-functions. We have already touched this issue in Section~\ref{choice_of_support}. In the following, we readdress the problem in our specific data-context: the AIC-optimal support-choice in Figure~\ref{support_analysis} is clear-cut. On the other hand, the linearity of the log/log-plot in Figure~\ref{loglog_plots} is also convincing. However, the consequences of infinite excitements as in \eqref{power_law} or \eqref{bm_est} with decay parameters of $\alpha = 1.1$ and lower are hard to interpret: in such a model, a significant part of the offspring happens days or weeks after the point of reference. In the case of \eqref{power_law}, where most of the instantaneous excitement is shifted to the tail, such a low decay-parameter would mean that a large part of the offspring happens only after 30'000 years! These are brave extrapolations when the size of the data window is $30\min$. So typically, one desires larger samples.
%
%
Such long order-book histories, however, contain trading halts, night times, weekends, holidays, and so on. Including these obvious regime-switches in the estimation will typically yield excitement estimates that are even more heavy-tailed. In autoregressive fits, stationary models typically exhibit long-range dependence when calibrated to time windows where there are regime switches in the data; see \citet{mikosch00}. That regime-switches have these effects in the point process context can easily be demonstrated in simulation, e.g., by calibrating a Hawkes model on data from a doubly stochastic Cox process as described in Section~\ref{diagnostics}. Also note that other unobserved covariate-processes might have simililar effects. E.g., when we fit a univariate Hawkes model to data that is a margin of a biviariate Hawkes model, the fit typically also exhibits long-range dependence. These issues make observed long-range dependence as in Figure~\ref{loglog_plots} hard to interpret. The interpretation of the results is finally a question of taste and, more importantly, a question of the application: in view of the argumentation in Section~\ref{choice_of_support}, the truncated model typically suffices if the goal is an algorithm that aims to calculate (as quickly as possible) the likelihood of a specific order-book event given the event-history (e.g., for the implementation of some strategy). Also, if we are mainly interested in the causal structure underlying data (which order flow affects which), the truncated models may suffice. On the other hand, if we are more interested in theoretical results such as connecting the market microstructure to coarser scale (price) processes, volatility estimation, and so on, then it is presumably more fertile to work with the extrapolation of the power-laws.
}

\subsubsection*{{Comparison with earlier results}}
{
The most complete Hawkes-process based limit order book model as of now is \citet{bacry14a}. Here, the authors distinguish the ask and the bid side of the order flow and also take price jumps and cancellations into account. The authors find that it is mainly the price jumps that drive the dynamics of the process. They also report a strong influence from trades on limit orders and less excitement vice versa, which supports our observations. Also the power-law shape with decay parameters around 1 of the excitements is reported in all papers where the Bacry--Muzy nonparametric method is applied to the order flow;
see, e.g., \citet{hardiman13} and \citet{bacry14a}.  As to now, there seems to be no approach for the choice of the support parameter $s$, respectively, $A$ in the BM-method. So the (relatively small) support estimator from our AIC-based selection cannot be compared with earlier results.
}

\begin{figure}
\center
\includegraphics[width = 0.8 \textwidth]{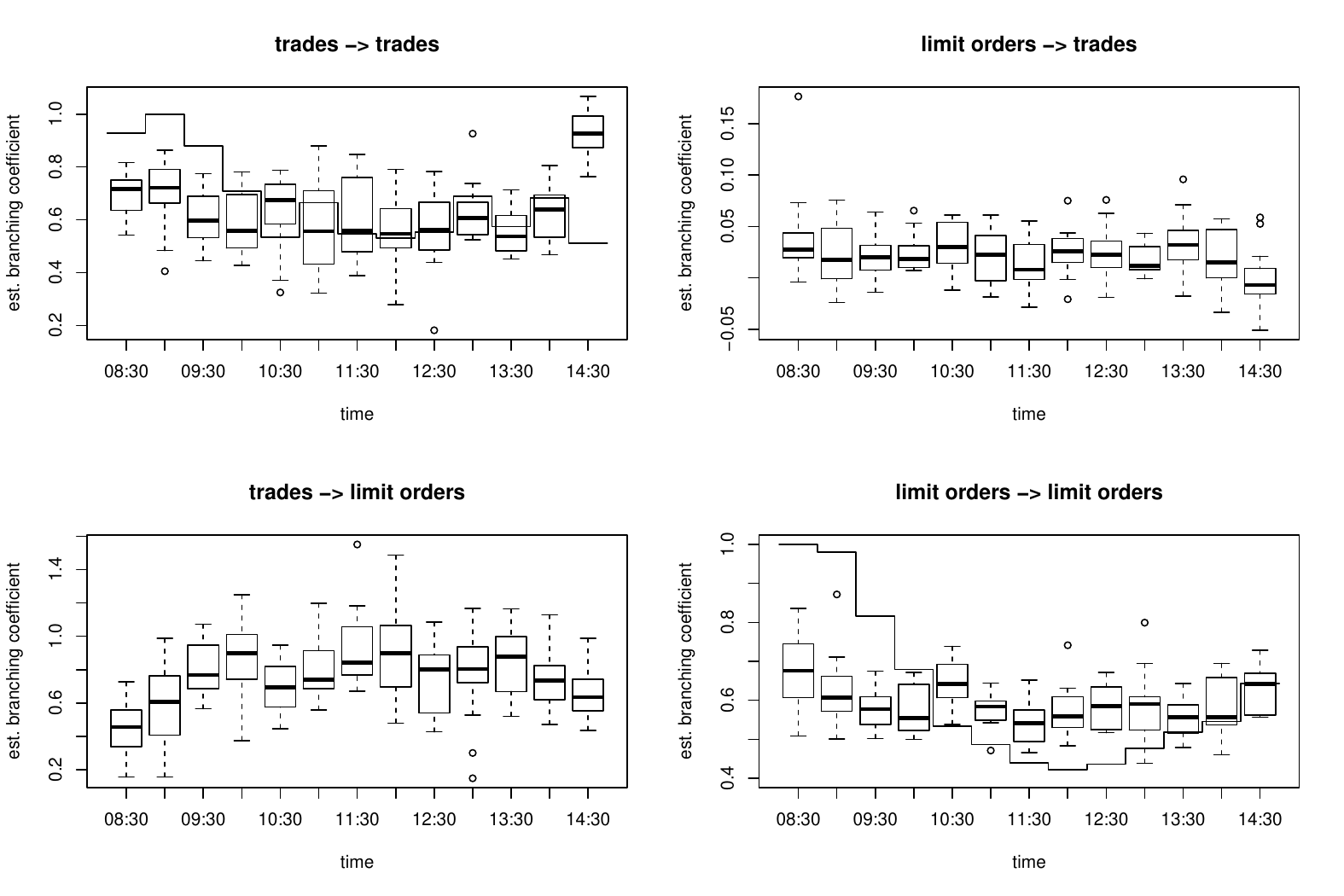}
\caption{Comparison of the estimates over the trading day; see Section~\ref{Bivariate_estimation}. For each of the fourteen sample days, we consider thirteen $30\min$ windows. The boxplots collect the branching-coefficient estimates (left scales); the baseline-intensity estimates are averaged over the corresponding time-windows (step functions). The branching coefficients are relatively stable over the trading day---with outliers at the opening and the close of the CME. The baseline intensity of the market orders reflects the U-shape of the average intensity from Figure~\ref{unconditional_intensity}.}\label{window_by_window}
\end{figure}
\subsection{Interpretation of the estimation results}\label{interpretation}
The interpretation of the estimation results from Section~\ref{Bivariate_estimation} (and of Hawkes fits in general) is not straightforward: observing the arrival of an order makes people (respectively, algorithms) send other orders. 
In this sense, we may expect some quite direct true excitement in LOB data. In the Hawkes  modeling approach however, any fluctuation of 
exogenous processes that influence the observed event-process will also be detected as excitement. {The past of the observed process then serves as a proxy for some unobserved covariate processes}. {E.g., the past price-jumps in \citet{bacry14a} is reported to be the most important `excitor' for the order flows. These price jumps can be seen as a proxy for a change in the state of the book:  updated---typically larger---volume at best bid and best ask, updated imbalance, and so on.} Candidates for other covariate processes in our context are volume, arrival of orders away from the best bid or best ask price, spread, or even data from other assets such as options on S\&P 500 E-mini futures. 
A most natural way to model this situation would be a joint multivariate Hawkes model. 
However, doing statistics with so little knowledge about the state (or even the dimensionality) of the process yields new problems. So the best way to get rid of artificial selfexcitement in the Hawkes model is presumably to make the baseline intensity more flexible. For an example of such a Hawkes model with stochastic baseline-intensity; see \citet{zhao12}. 
To summarize: our estimation method can indeed detect self- or crossexcitement in data. However, we ought to be careful with interpretation of these terms.
\par{Any} Hawkes fit is meaningful and fertile despite of the criticism above and despite of the vanishing \emph{p}-values in our application: plots of excitement estimates as in Figure~\ref{td2} 
are visualizations of huge event-data sets in a compact and at the same time informative way. In that sense, any Hawkes fit---and our estimation method in particular---can be used as a graphical tool for exploratory event-stream analysis. Furthermore, even if the Hawkes model assumption may be completely wrong for the data-generating point process $\bfN$, an excitement-function estimate $\hat{h}_{ij}(\cdot)$ is still meaningful. It is an estimate for the best linear filter of `$\E\left[N_i(\d t)/\mathrm{d}t| \sigma\left(N_j(\{s\}), s < t\right)\right]$' which is a relevant quantity in all stationary models. 
\section{Conclusion}
This paper demonstrates that applying methods from time series theory to the bin-count sequences of point process data yields a useful and intuitive nonparametric estimation method for the multivariate Hawkes process. The price for the fertile simplicity of the method is a bias due to {various errors} involved in the approximation. Simulation studies support that this bias can be controlled and that it is negligible for most practical means. The technique presente	d depends on the choice of the bin size and the assumed support of the excitement function(s). Methods for a sensible choice of these parameters are given. In any application, the robustness with respect to these choices ought to be studied. {We treat computational issues in high-dimensional cases in \citet{kirchner16a}.  A larger limit-order book application will be given in \citet{kirchner16b}, where we consider an application of our method on Hawkes models with marks and with covariate-dependent baseline intensities.} 

\par Finally, note that in view of the analogy between discrete-time INAR($p$) sequences and continuous-time Hawkes processes, analysts using the Hawkes model may consider to directly apply the INAR($p$) model in the first place---as most event data live on relatively discrete time grids. 
\appendix
\section*{Acknowledgements}  M.K.\ is indebted to Paul Embrechts for guidance and support during the preparation of the paper.
The author acknowledges financial support from ETH RiskLab and the Swiss Finance Institute. Furthermore, M.K.\ thanks Robert Almgren for sharing his expertise on limit-order-book data, Marius Hofert as well as Martin Maechler \citep{maechler12} for support with {\tt R}, Val\'{e}rie Chavez-Demoulin as well as Thibault Vatter for numerous comments on an earlier versions of the paper, and Rita Kirchner as well as Anne MacKay for help with the editing. Incorporating the comments of two anonymous referees made the paper more complete.

\bibliographystyle{apalike}
\bibliography{diss.biblio.bib}

\appendix

\section{Proofs} \label{proofs}
\subsection{Proof of Proposition~\ref{white_noise}} \label{pf_white_noise}
First, we establish that $\bfu_n:= \bfX_n - \bfa_0 - \sum_{k=1}^p A_k\bfX_{n-k},\, n\in\Z,$ defines a white noise sequence. Stationarity of $\left(\bfu_n\right)$ follows from the stationarity of $\left(\bfX_n\right)$.
For the sequel of the proof, fix any $n\in\Z$. 
{
Denote $
\mathcal{F}_n :=
 \sigma \{\bfX_k:k\leq n\}$.
Note that $\E\left[A_k\circledast \bfX_{n-k}| \mathcal{F}_n\right]=A_k \bfX_n,\, A_k\in\R^{d\times d}_{\geq 0},$ and that $\bfepsilon_n$ is independent of  $\mathcal{F}_n$. So we get
\begin{align}
\E \big[\bfu_n |\mathcal{F}_{n -1} \big]=  \E\big[\sum\limits_{k=1}^p A_k \circledast \bfX_{n-k} + \bfepsilon_n -\bfa_0 - \sum_{k=1}^\infty A_{k} \bfX_{n-k}|\mathcal{F}_{n - 1}\big] = 0, \label{cond_exp}
\end{align}
and, consequently, $\E \bfu_n = 0,\, n\in\Z$.
For the autocovariances of the errors, note that, for $n'<n$ (and then, by symmetry, for $n'\neq n$),
\begin{align*}
\E\big[\bfu_n \bfu_{n'}\big]&= \E\Big[\E\big[\bfu_n \bfu_{n'}|\mathcal{F}_{n-1}\big]\Big] =  \E\Big[\bfu_{n'}\underbrace{\E\big[\bfu_n |\mathcal{F}_{n-1}\big]}_{\stackrel{\eqref{cond_exp}}{=} 0}\Big] = 0.
\end{align*}
\noindent Finally, we have that
\begin{align*}
\Cov(\bfu_n) & = \E\Big[\underbrace{\Cov\big(\bfu_n|\mathcal{F}_{n-1}\big)}_{\stackrel{}{=}\Cov(\bfX_n| \mathcal{F}_{n-1})}\Big] + \Cov\Big(\underbrace{\E\big[\bfu_n|\mathcal{F}_{n-1}\big]}_{\stackrel{\eqref{cond_exp}}{=} 0}\Big) 
= 
\diag\Bigg(\bfa_0 +\sum\limits_{k = 1}^p A_k \E\big[\bfX_{n-k}\big]\Bigg)  \\
&= \diag\Bigg(\bfa_0 +\sum\limits_{k = 1}^p A_k \Big(1_{d\times d} - \sum\limits_{k = 1}^p A_k \Big)^{-1}\bfa_0\Bigg) 
=  \diag\Bigg(\Big(1_{d\times d} - \sum\limits_{k = 1}^p A_k \Big)^{-1}\bfa_0\Bigg).
\end{align*}
\begin{flushright}
{ $\square$}
\end{flushright}
}

\subsection{Proof of Theorem~\ref{inference}} \label{pf_inference}
The first part of the proof largely depends on matrix manipulations. So it is important to remind the reader that all vectors are understood as column vectors.
We rewrite the INAR($p$) sequence $(\bfX_k)\subset\N_0^d$ as a standard multivariate linear autoregressive time series with white-noise error sequence $(\bfu_k)_{k\in\Z} := (\bfX_k-\mathbf{a}_0 + \sum_{l=1}^p A_l \bfX_{k-l})_{k\in\Z}$
$$
\bfX_k = \mathbf{a}_0 + \sum\limits_{l=1}^p A_l \bfX_{k-l} + \bfu_k,\quad k\in\Z;
$$
see Corollary~\ref{AR-representation}. Then the distributional properties of the CLS-estimator are derived similarly as in \citet{luetkepohl05}, pages 70--75, where independent errors are assumed.
In the following, let $\bfZ\in\N_0^{(dp+1)\times (n-p)}$ be the design matrix from the CLS Definition~\ref{CLS} with respect to the sample $\left(\bfX_1,\bfX_2,\dots,\bfX_n\right)$.  Furthermore, let
$
\mathbf{U}:=(\bfu_{p+1},\bfu_{p+2},\dots, \bfu_{n})\in\R^{d\times (n-p)}.
$
Note that $\bfZ$ as well as $\mathbf{U}$ depend on $n$.
We work under the assumption that
\begin{align}
&\frac{1}{n-p}\bfZ\bfZ^\top\stackrel{p}{\longrightarrow}:\Gamma\in\R^{{(dp+1)\times (dp+1)}},\quad n\longrightarrow \infty,\quad
\label{A1}
\end{align}
\text{exists and is invertible.}
In addition, we use that, for $n\longrightarrow\infty$,
\begin{align}
 \frac{1}{\sqrt{n-p}} \vec\big(\mathbf{U}\bfZ^\top\big)
 &\stackrel{\d}{\longrightarrow} \mathcal{N}_{{d^2p+d}}\Bigg(0_{{d^2p+d}}, \E\left[\Big(\bfZ_0\otimes 1_{d \times d}\Big)
\bfu_0\Big(\big(\bfZ_0\otimes 1_{d \times d}\big)\bfu_0\Big)^\top \right] \Bigg), \label{A2}
 \end{align}
where
 $\bfZ_0:=  \left(\bfX^\top_{-1}, \bfX^\top_{-2},\dots ,\bfX^\top_{-p},1\right)^\top\in \N_0^{(pd+1)\times 1}$ has the same distribution as any of the columns of the design matrix $\bfZ$.
We postpone the reasoning for \eqref{A2} to the end of the proof. As a first step, weak consistency of $\hat{\bfB}^{(n)}\in\R^{d\times(dp+1)}$ is proven. To that aim, we will use that
\begin{align}
 \bfY := \left(\bfX_{p+1},\bfX_{p+1},\dots,\bfX_{n}\right) &= \bfB\bfZ + \mathbf{U}\left(\in\N^{d\times (n-p)}_0\right);\label{bfY} 
 \end{align}see Definition~\ref{CLS}.
\begin{align*}
 \hat{\bfB}^{(n)} - \bfB  =  \bfY\bfZ^\top \left(\bfZ\bfZ^\top\right)^{-1} - \bfB
\stackrel{\eqref{bfY}}{=}  \left(\bfB\bfZ + \mathbf{U}\right)\bfZ^\top \left(\bfZ\bfZ^\top\right)^{-1} - \bfB
&=  \mathbf{U}\bfZ^\top \left(\bfZ\bfZ^\top\right)^{-1} \\
&=  \frac{\mathbf{U}\bfZ^\top}{n-p}\left(\frac{\bfZ\bfZ^\top}{n-p}\right)^{-1} .
\end{align*}
By \eqref{A1}, the second factor converges in probability to the constant matrix $\Gamma$. By \eqref{A2}, the first factor has the same asymptotic distribution as $\tilde{W}/{\sqrt{n-p}}$ where $\tilde{W}$ is a matrix consisting of jointly normally distributed entries not depending on $n$. So $\tilde{W}/{\sqrt{n-p}}\stackrel{p}{\longrightarrow} 0_{d \times (dp +1)}$ and therefore  $\hat{\bfB}^{(n)} - \bfB\stackrel{p}{\longrightarrow} 0_{d \times (dp +1)}$.
For establishing the asymptotic distribution, we treat the difference of the estimated and true vectorized parameter-matrix in a similar way:
\begin{align}
\vec\left(\hat{\bfB}^{(n)}\right) - \vec\left(\bfB\right)
=  \vec\left(\hat{\bfB}^{(n)} - \bfB\right)
&= \vec\left(\mathbf{U}\bfZ^\top \left(\bfZ\bfZ^\top\right)^{-1}\right)\nonumber\\
&= \left(\left(\bfZ\bfZ^\top\right)^{-1} \otimes 1_{d\times d}\right)\vec\left(\mathbf{U} \bfZ^\top\right)\nonumber \\
 &= \frac{1}{\sqrt{n-p}}\left(\left(\frac{\bfZ\bfZ^\top}{n-p}\right)^{-1} \otimes 1_{d\times d}\right)\vec\left(\frac{\mathbf{U} \bfZ^\top}{\sqrt{n-p}}\right)\label{prod}.
  \end{align}
In the third step of the calculation above we use that
\begin{equation}
\vec\left(AB\right) =  \left(B^\top \otimes I\right) \vec\left(A\right), \label{kronecker_rule}
\end{equation} for matrices $A, B$ and identity matrix $I$ such that the calculations are consistent dimensionwise; see A.12 in \citet{luetkepohl05}.
It follows from \eqref{prod} together with \eqref{A1} that
 $\sqrt{n-p} \left(\vec\left(\hat{\bfB}^{(n)}\right)- \vec\left(\bfB\right)\right)$ has the same asymptotic distribution as
\begin{equation}\label{aux2}
\left(\Gamma^{-1}\otimes  1_{d\times d}\right)\vec\left(\frac{\mathbf{U} \bfZ^\top}{\sqrt{n-p}}\right).
\end{equation}
With \eqref{A2}, we then find that the asymptotic distribution of \eqref{aux2}---and therefore of $\sqrt{n-p} \left(\vec\left(\hat{\bfB}^{(n)}\right)- \vec\left(\bfB\right)\right)$---is centered normal with covariance matrix
\begin{align*}
&\left(\Gamma^{-1}\otimes  1_{d\times d}\right) \mathrm{cov}\left(\dlim\limits_{n\to\infty}\vec\left(\frac{\mathbf{U} \bfZ^\top}{\sqrt{n-p}}\right) \right)\left(\Gamma^{-1}\otimes  1_{d\times d}\right)\\
& \stackrel{\eqref{A2}}{=}\left(\Gamma^{-1}\otimes  1_{d\times d}\right) \E\left[\left(\bfZ_0\otimes 1_{d \times d}\right)\bfu_k\left(\left(\bfZ_0\otimes 1_{d \times d}\right)\bfu_k\right)^\top \right]\left(\Gamma^{-1}\otimes  1_{d\times d}\right) .\\
\end{align*}
We still have to establish \eqref{A2}. To that aim, we rewrite the left-hand side of \eqref{A2} as
\begin{align*}
\frac{1}{\sqrt{n-p}} \vec\left(\mathbf{U}\bfZ^\top\right) &=\frac{1}{\sqrt{n-p}}\vec\left(\left(\sum\limits_{j=1}^{n-p}\bfZ^\top_{j,1}\mathbf{U}_{\cdot, j},\dots, \sum\limits_{j=1}^{n-p}\bfZ^\top_{j, dp+1}\mathbf{U}_{\cdot, j}\right)\right)\\
& = \frac{1}{\sqrt{n-p}}\sum\limits_{j=1}^{n-p}\vec\left(\left(\bfZ_{1,j}\mathbf{U}_{\cdot ,j},\dots,\bfZ_{dp+1,j}\mathbf{U}_{\cdot, j}\right)\right)\\
& = \frac{1}{\sqrt{n-p}}\sum\limits_{j=1}^{n-p}\vec\left(\left(\mathbf{U}_{1, j},\dots,\mathbf{U}_{{d}, j}\right)^\top\left(\bfZ_{1,j},\bfZ_{{2},j},\dots,\bfZ_{dp+1,j} \right)\right)\\
& =\frac{1}{\sqrt{n-p}}\sum\limits_{k=p+1}^{n}\vec\left(\bfu_k \cdot \bfZ_k^\top\right),
\end{align*}
where $ \bfZ_k:= \left(\bfX^\top_{k-1}, \bfX^\top_{k-2},\dots ,\bfX^\top_{k-p},1\right)^\top\in \N_0^{(pd+1)\times 1}$. Note that, for $k\in\{p+1,\dots,n\}$, $\bfZ_k$ is the $(k-p)$-th column of the design matrix $\bfZ$. Now, let $
 \bfw_k:=\vec\left(\bfu_k \cdot \bfZ_k^\top\right)\in\R^{pd^2+d},\, k\in\Z$.
We show that for the sequence $( \bfw_k)\subset \R^{pd^2+d}$, a central limit theorem for vector-valued martingale differences can be applied. Proposition 7.9 from \citet{hamilton94} states that if $\left(\bfw_k\right)\subset\R^{\tilde{d}}$ is such that 
\begin{itemize}
\item[(a)] it defines a vector-valued martingale difference sequence, i.e., there is a filtration $\left(\mathcal{H}_k\right)_{k= p+1, p+2, \dots ,n}$ such that $\bfw_k$ is $\mathcal{H}_k$-measurable and $\E\left[\bfw_k|\mathcal{H}_{k-1}\right] = 0_{\tilde{d}}$,\, $k\in\Z$,
\item[(b)] $\E\left[\bfw_{k}\bfw_k^\top\right]=:S\in\R^{\tilde{d}\times\tilde{d}}$ is a positive definite matrix independent of $k$,
\item[(c)]  for all $k_1,k_2,k_3,k_4\in\Z$ and for all $i_1,\dots, i_4\in\{1,2,\dots,\tilde{d}\}$, $$\E\left[\bfw_{k_1,i_1}\bfw_{k_2,i_2}\bfw_{k_3,i_3}\bfw_{k_4,i_4}\right]< \infty,$$ where $\bfw_{k,i}$ denotes the $i$-th component of $\bfw_{k}$, and
\item[(d)] $\sum\limits_{k=p+1}^n\frac{1}{n-p}\bfw_k\bfw_k^\top \stackrel{p}{\longrightarrow}  S$,
\end{itemize}
then, for $n\longrightarrow \infty$,
$
{1}/{\sqrt{n-p}}\sum_{k=p+1}^n\bfw_k \stackrel{\mathrm{d}}{\longrightarrow} \mathcal{N}_{\tilde{d}}(0_{\tilde{d}}, S).
$

\vspace{0.3cm}
\strut\\ \emph{Proof of (a)} Define the filtration $\left(\mathcal{H}_k\right)$ by setting $$\mathcal{H}_k:=\sigma\big(\left(\bfu_i, \bfX_{i-1},\bfX_{i-2}, \dots,\bfX_{i-p}\big):\ i\leq k\right),\,k\in\Z .$$ Then one can easily check that $\bfw_k=\vec\left(\bfu_k \cdot \bfZ_k^\top\right)$ is  $\mathcal{H}_k$- measurable. 
It suffices to prove the martingale-difference property for the sequence $\left(\bfu_k\right)$ since $\bfX_{k'}$ for $k'<k$ and therefore $\bfZ_k$ are $\mathcal{H}_{k-1}$-measurable. But then, because 
$$
\E\left[\bfX_k \Big|\mathcal{H}_{k-1}\right]= \E\left[\bfepsilon_k +  \sum\limits_{m = 1}^p A_m \circledast \bfX_{k-m} \Big|\mathcal{H}_{k-1}\right] = \bfa_0 +  \sum\limits_{m = 1}^p A_m \bfX_{k-m},
$$
we obtain the martingale difference property:
\begin{align*}
\E\left[\bfu_k\big|\mathcal{H}_{k-1}\right] 
&= \E\left[\bfX_k - \bfa_0 - \sum\limits_{m = 1}^p A_m \bfX_{k-m}\Big|\mathcal{H}_{k-1}\right]
=\E\left[\bfX_k |\mathcal{H}_{k-1}\right]- \bfa_0 - \sum\limits_{m = 1}^p A_m \bfX_{k-m}
= 0_d.
 \end{align*}
\emph{Proof of (b)} Independency of $k$ follows from stationarity of $\left(\bfw_k\right)$. Choose $k=0$.
We need to show that, for $b \in \R^{d(pd+1)}\setminus \{ 0_{d(pd+1)}\}$,
\begin{equation}
b^\top \E\left[\bfw_k\bfw_k^\top\right]b = \E\left[b^\top \bfw_0\bfw_0^\top b^\top \right] = \Var\left(b^\top \bfw_0\right)>0.\label{claim}
\end{equation}
With \eqref{kronecker_rule}, we find
\begin{align}
\bfw_0 = \vec\left(\bfu_0 \cdot\bfZ_0^\top\right)=   \left( \bfZ_0\otimes 1_{d\times d}\right)\vec\left(\bfu_0\right)
= \left(\bfZ_0\otimes 1_{d \times d}\right)\bfu_0\label{bfw0}
\end{align}
and therefore
\begin{align*}
 \E\left[\bfw_0 \bfw_0^\top\right] 
 =  \E\left[\left(\ \bfZ_0\otimes 1_{d \times d}\right)\bfu_k\left(\left( \bfZ_0\otimes 1_{d \times d}\right)\bfu_0\right)^\top\right]
  =  \E\Big[\left(\bfZ_0\otimes 1_{d \times d}\right)\bfu_0\bfu_0^\top\left(\bfZ_0\otimes 1_{d \times d}\right)\Big].
\end{align*}
To establish \eqref{claim}, we define the $\sigma$-algebra 
$$
\mathcal{F}:= \sigma\left(\bfX_{-1}, \dots ,\bfX_{-p}, A_1\circledast\bfX_{-1}, \dots ,A_p\circledast\bfX_{-p}\right).
$$
Note that $\bfZ_0$ is $\mathcal{F}$-measurable and $\bfepsilon_0$ is independent of $\mathcal{F}$. Using these facts when considering the expectation of the conditional variance of $b^\top \bfw_0$, we obtain
\begin{eqnarray}
\Var\left(b^\top \bfw_0\right)& = &\E\left[\Var\left(b^\top \bfw_0|\F\right) \right]  +  \Var\left(\E\left[b^\top \bfw_0|\F\right]\right) \nonumber\\
 &\geq&   \E\left[\Var\left(b^\top \bfw_0|\F\right) \right] \nonumber \\
 &\stackrel{\eqref{bfw0}}{=} &  \E\left[\Var\left(b^\top \left(\bfZ_0\otimes  1_{d\times d}\right)\bfu_0|\F\right) \right] \nonumber \\
  &=&   \E\left[b^\top \left(\bfZ_0\otimes  1_{d\times d}\right)\Cov\left(\bfu_0|\F\right) \left(b^\top \left(\bfZ_0\otimes  1_{d\times d}\right)\right)^\top\right]\label{conditional_variance}.
 \end{eqnarray}
Since
\begin{align}
\bfu_0 
 = \bfX_0 - \bfa_0 - \sum\limits_{i=1}^p A_i\bfX_{-i}
 = \bfepsilon_0 + \sum\limits_{i=1}^p A_i\circledast \bfX_{-i} - \bfa_0 - \sum\limits_{i=1}^p A_i\bfX_{-i},\label{bfu}
\end{align}
the summand $\bfepsilon_0$ is the only term that contributes to the conditional covariance matrix in \eqref{conditional_variance}---the other summands in \eqref{bfu} are constant with respect to $\mathcal{F}$ and $\bfepsilon_0$ is independent of $\mathcal{F}$. So we have $\Cov\left( \bfu_0|\F\right)= \Cov\left(\bfepsilon_0|\F\right)= \Cov\left(\bfepsilon_0\right)=\diag(\bfa_0)$ and continuing with \eqref{conditional_variance} we find 
\begin{align}
\Var\left(b^\top \bfw_0\right) &\geq   \E\left[b^\top \left(\bfZ_0^\top\otimes  1_{d\times d}\right)
\diag\left(\bfa_0\right)
\Big(b^\top \left(\bfZ_0^\top\otimes  1_{d\times d}\right)\Big)^\top\right]\nonumber\\
&=
\E\left[\sum\limits_{i=1}^{d}\bfa_{0,i}\Big(b^\top \left(\bfZ_0^\top \otimes  1_{d\times d}\right)\Big)_{1,i}^2\right]\nonumber\\
&\geq
\bfa_{0,i_0}\E\left[\Big(b^\top \left(\bfZ_0^\top \otimes  1_{d\times d}\right)\Big)_{1,i_0}^2\right]>0\label{bfa_margin},
\end{align}
where $i_0\in\{1,2,\dots,d\}$ in \eqref{bfa_margin} is  chosen in such a way that $\bfa_{0,i_0}>0$. (Remember that $\bfa_0\neq 0_d$, by assumption.)
The strict inequality in \eqref{bfa_margin} follows because, for $j_0\in\{1,2,\dots,np+d\}$ such that $b_{j_0}\neq 0$, we have that
\begin{align*}
\P\left[\Big(b^\top \left(\bfZ_0^\top\otimes  1_{d\times d}\right)\Big)_{1,i_0} \neq 0 \right] 
=\P\left[b^\top \cdot\left(\bfZ_0^\top\otimes  1_{d\times d}\right)_{\cdot,i_0} \neq 0 \right] 
\geq \P\big[b_{j_0}\bfX_{k_0,{l_0}} \neq 0\big]
&=  \P\left[\bfX_{k_0,{l_0}} \neq 0\right]  \\
&> 0,
\end{align*}
for some $k_0\in\Z$ and some ${l_0}\in\{1,2,\dots,d\}$ dependent on $j_0$. Note that
$\bfX_{k,{l}}$ denotes the $l$-th component of $\bfX_k$.
By stationarity, $k_0\in\Z$ is irrelevant. And the case that $\bfX_{0,l_0}  = 0\ a.s.$ for some  $l_0\in\{1,2,\dots,d\}$ we have excluded, so the strict inequality follows.\\
\vspace{0.5cm}

\emph{Proof of (c)} Note that claim $(c)$ follows if
$
\E\left[\bfX_{k_1, i_1} \cdots \bfX_{k_8, i_8}\right]<\infty$ for $ k_1,\dots,k_8\in\Z,\, i_1,\dots, i_8 \in\{1,2,\dots, d\}.
$
The boundedness of these expectations is established for the univariate case in Corollary 1 of \citet{kirchner16c}. For the multivariate case, one can argue similarly via the existence of the moment generating function in a neighborhood of zero.\\
\vspace{0.2cm}\strut\\
\emph{Proof of (d)} We show that $\left(\bfw_k\bfw_k^\mathrm{T}\right)$ is ergodic. Then the claim of $(d)$ follows with the Birkhoff-Khinchin Ergodic Theorem. The sequence $\left(\bfX_k\right)$ can be represented as margin of a $pd$-dimensional INAR$(1)$ sequence $\big(\tilde{\bfX}_k\big)$; see \citet{latour97}. It is easily checked that the latter is an irreducible, aperiodic Markov chain on $\mathbb{N}_0^{pd}$. So $\big(\tilde{\bfX}_k\big)$ is ergodic; see  \citet{durrett95}, page 338. As margins of ergodic processes are ergodic, $\left(\bfX_k\right)$ also is ergodic. 
As $\bfw_k$ can be written as a measurable function of the past of $\left(\bfX_k\right)$, $\left(\bfw_k\right)$ is also ergodic.
Finally, $\left(\bfw_k\bfw_k^\top\right)$ is ergodic because it is a measurable transformation of the ergodic sequence $\left(\bfw_k\right)$.
\begin{flushright}
{ $\square$}
\end{flushright}
\end{document}